\documentclass{gtpart}     
\usepackage{graphicx}  
\title{Roots, symmetries and conjugacy of pseudo-Anosov mapping classes.}

\author[J Fehrenbach]{J\'er\^ome Fehrenbach}
\givenname{J\'er\^ome}
\surname{Fehrenbach}
\address{Institut de Math\'ematiques de Toulouse\\31062 TOULOUSE CEDEX 09\\France }
\email{jerome.fehrenbach@iut-tlse3.fr}
\urladdr{}

\author[J Los]{J\'er\^ome Los}
\givenname{J\'er\^ome}
\surname{Los}
\address{Universit\'e de Provence, LATP, UMR CNRS 6632\\
39 Rue F. Joliot Curie\\
13453 Marseille cedex 13\\
France}
\email{los@cmi.univ-mrs.fr}
\urladdr{}


\arxivreference{}
\arxivpassword{}

\volumenumber{}
\issuenumber{}
\publicationyear{}
\papernumber{}
\startpage{}
\endpage{}
\doi{}
\MR{}
\Zbl{}
\received{}
\revised{}
\accepted{}
\published{}
\publishedonline{}
\proposed{}
\seconded{}
\corresponding{}
\editor{}
\version{}

\newtheorem{thm}{Theorem}[section]    
\newtheorem{proposition}{Proposition} 
\newtheorem{lem}[thm]{Lemma}          
\theoremstyle{definition}
\newtheorem{defn}[thm]{Definition}    
\begin{document}

\begin{abstract}    
An algorithm is proposed that solves two decision problems for pseudo-Anosov elements in the mapping class group of a surface with at least one marked fixed point. The first problem is the root problem: decide if the element is a power and in this case compute the roots. The second problem is the symmetry problem: decide if the element commutes with a finite order element and in this case compute this element. The structure theorem on which this algorithm is based provides also a new solution to the conjugacy problem.
\end{abstract}

\maketitle


\section{Introduction}

We solve two algorithmic problems about pseudo-Anosov elements in the mapping class group of punctured surfaces with at least one fixed puncture. 

\begin{thm} 
\label{th:1.1}
Let $S$ be a surface with $n+1, n \geqslant 0$ marked points 
$\{x_0, x_1, ..., x_n\}$ and $[f]$ a pseudo-Anosov mapping class fixing $x_0$ and $\{x_1, ..., x_n\}$ setwise.\\
There exists a finite algorithm to decide whether $[f]$ is a power and in this case the root is computed.\\
There exists a finite algorithm to decide whether $[f]$ commutes with a finite order element fixing $x_0$ and which computes it if it exists.
\end{thm}

The combinatorial tools that are developed for these questions give also, as a by product, a new solution to the conjugacy problem among pseudo-Anosov elements. \\
The restriction to the class that fixes a marked point is certainly not optimal but it makes the description of the solution and the arguments much
simpler.  The conjugacy problem for pseudo-Anosov elements in the mapping class group is known from Hemion \cite{he79}, another solution was given later by Mosher \cite{mo86}. A general solution for the mapping class group, ie including reducible elements, has been given in Keil's thesis \cite{ke97}. The solution proposed here is easily computable.
Theorem \ref{th:1.1} is based on a structure theorem (theorem \ref{th:5.1}) that describes a canonical subset of the set of {\em efficient representatives} of $[f]$
as developed by Bestvina and Handel \cite{BH95} for surface automorphisms and based on the free group train track representatives the same authors introduced in \cite{BH92}, with some previous weak versions by Franks and Misiurewicz \cite{FM93} and Los \cite{lo93}. 

The structure theorem \ref{th:5.1} describes a canonical  set associated to $[f]$ called {\em single cut (SC) representatives} of $[f]$ as a union of "cycles". This set is a complete conjugacy invariant. Most of the paper is about defining, characterising and computing these cycles.
The structure of these cycles is closely related with the local property of the pseudo-Anosov homeomorphism $f$ around the fixed marked point $x_0$.

Questions about roots of pseudo-Anosov homeomorphisms have attracted some attention lately, see for instance Gonzales-Meneses \cite{GM03} in the case of the punctured discs and Bonatti and Paris \cite{BP07} for more general surfaces.

This paper is a rewriting, with complements, of a part of the first author's thesis \cite{fe98}, defended in 1998 and advised by the second author.
The second author would like to thank Luis Paris and Bert Wiest who motivated him to start the rewriting. 

\section{Efficient representatives on surfaces}

We consider a pseudo-Anosov homeomorphism $f$ on a compact surface $S$ with $n+1$ marked points $X = \{ x_0, x_1, ..., x_n \}, n\geqslant 0 $ among which $x_0$ is fixed. In this section we review definitions and properties of {\em efficient representatives} for $[f]$ the isotopy class of $f$ on $S_X = S - X$. For a more detailed presentation and examples, the reader is referred to Bestvina and Handel \cite{BH95}, and for an implementation of the algorithm to Brinkmann  \cite{Br}.

\subsection{Combinatorial aspects.}

\subsubsection{Topological representatives.} 

We start with some preliminary and general definitions. A {\em topological representative} of $[f]$ is a triple $(\Gamma, \Psi, h_{\Gamma})$ where:\\
- $\Gamma$ is a graph without valence one vertices,\\
- $h_{\Gamma}: \Gamma \rightarrow S_X$ is an embedding that induces an homotopy equivalence,\\
- $\Psi: \Gamma \rightarrow \Gamma$ is a combinatorial map that represents $f$.

Let us make this definition more explicit. In all the paper graphs will be considered from several point of views. Combinatorially a finite graph is given by a collection of vertices and oriented edges, that are denoted respectively $V(\Gamma)$ and $E(\Gamma)$. The reverse of the edge $e$ is denoted $\overline{e}$. Graphs will also be considered as topological or metric spaces depending on the context.\\
A pair $(\Gamma, h_{\Gamma})$ is called an {\em embedded graph}, it contains the following data:\\
(a) The embedding implies the existence of a cyclic ordering for the edges that are incident at a vertex $v\in V(\Gamma)$, ie the star $St(v)$.\\
(b) The  fundamental groups  $\pi_1(\Gamma)$ and $\pi_1(S_X)$ are isomorphic.\\
(c) A regular neighbourhood $N^{reg}[ h_{\Gamma} (\Gamma)]$ is homeomorphic to the surface:
 $\hat{S}_X := clos( S - \bigcup D_{x_i} )$, where the $ D_{x_i}$ are small disjoint discs centred at the marked points.\\
The map $\Psi: \Gamma \rightarrow \Gamma$ is an endomorphism that satisfies:\\
- $\Psi (V(\Gamma)) \subset V(\Gamma) $,\\
- for every edge $e\in E(\Gamma),\ \Psi(e)$ is an edge path in $\Gamma$, it is represented by a word in the alphabet $\{ E(\Gamma) ^{\pm 1 } \}$.\\
The action of $\Psi$ on $\Gamma$ is homotopic, through $h_{\Gamma}$, to the action of $f$ on $S_X$, ie:\\
     $ (\ast)  \hspace{5cm} f \circ h_{\Gamma} \simeq h_{\Gamma} \circ \Psi  .$

To any topological representative  $ (\Gamma, \Psi, h_{\Gamma})$ is associated an {\em incidence matrix} $M$ whose entries are labelled by the edge set $E(\Gamma)$. The entry in place $(a,b)$ is the number of occurrences of the letters $a$ or $\overline{a}$ in the word $\Psi(b)$. This matrix depends only on the pair $ (\Gamma, \Psi)$.

{\em Remark:} 
A given isotopy class $[f]$ admits many topological representatives. For instance take a base point $y$ on $S_X$ and choose a set of generators for $\pi_1(S_X, y)$. This set is topologically represented by a bouquet of circles embedded in $S_X$, it defines a pair 
$(\Gamma, h_{\Gamma})$. Now consider any induced map 
$f_{\#}: \pi_1(S_X, y) \rightarrow  \pi_1(S_X, y)$ on this set of generators, this is a map $\Psi$ such that $(\Gamma, \Psi, h_{\Gamma})$ is a topological representative of $f$.

\subsubsection{Efficient representatives.}

The definitions are given in the general case ($n\geqslant 0$) and will be specified in the particular case $n = 0$. These two situations are quite different at the group automorphism point of view. Indeed the induced map $f_{\#}$ is an automorphism of the free group $\pi_1(S_X) = \pi_1(\Gamma)$ that is reducible when $n > 0$  and irreducible if $n=0$, following the terminology of Bestvina and Handel \cite{BH92}.\\
 Two edges $(a,b)$ of $\Gamma$ that are incident at a vertex $v$ and adjacent, according to the cyclic ordering (property (a) above), is a {\em turn}. A turn is {\em illegal} if there is an integer $k\geqslant 0$ such that 
$\Psi^k (a)$ and $\Psi^k (b)$ have a non trivial initial common edge path, where initial is understood with the orientation so that $v$ is the initial vertex.
The smallest such $k$ is called the {\em order} of the illegal turn. Turns are also cyclically ordered and a maximal set of consecutive illegal turns is a {\em gate}.
\begin{defn}
A triple $ (\Gamma, \Psi, h_{\Gamma})$ is an {\em efficient representative} of $[f]$ if the following conditions (G1)--(G5) are satisfied.
\end{defn}

(G1) $ (\Gamma, \Psi, h_{\Gamma})$ is a topological representative of $[f]$.\\
(G2) For every edge $e$ and every integer $k>0$ the edge path $\Psi^k (e)$ does not backtrack.\\
This means that all the words $\Psi^k (e)$ are reduced, ie no occurrences of subwords $a\overline{a}$.

Since $h_{\Gamma}(\Gamma)$ is homotopy equivalent to $S_X$, each distinguished point $x_i$ is contained in a disc bounded by a loop in $h_{\Gamma}(\Gamma)$.
Each such loop is represented by a closed edge path $b_i$ in $\Gamma$. 
The edges that belongs to the $b_i, i\geqslant 1$ are called {\em peripheral}. The union of the peripheral edges is a subgraph $P$ of $\Gamma$.
The {\em pre-peripheral} edges are the edges of $\Gamma$ that are eventually mapped to a path included in $P$ by $\Psi$. We denote $Pre P$ the union of these pre-peripheral edges and by $H$ the complementary graph: $ H = \Gamma - \left(P \cup Pre P\right)$.

(G3) The loops $b_i ,  i\geqslant 1$ are permuted under $\Psi$ and 
$\Psi|_P$ is a simplicial homeomorphism (the image of each edge in $P$ is a single edge in $P$).

If condition (G3) is satisfied then an illegal turn cannot consists of two peripheral or pre-peripheral edges. In addition the incidence matrix has the following form:
$$\left(
\begin{array}{ccc}
N & A & B \\
0 & C & D \\
0 & 0 & M_H
\end{array}
\right),$$
where the block $N$ corresponds to the peripheral edges and is a permutation matrix. The block $C$ corresponds to the pre-peripheral edges, it is a nilpotent matrix. The block $M_H$ corresponds to the edges of $H$.

(G4) The matrix $M_H$ is irreducible.\\
(G5) If the turn $(a,b)$ is illegal of order one then the first letter of $\Psi(a)$  (and $\Psi(b)$) is an edge of $H$.

{\em Remark:} If $n =0$ ie $S_X$ has only one marked point then the set of peripheral and pre-peripheral edges is empty, thus conditions (G3) and (G5) are vacuous and the incidence matrix is  $M_H$.

\subsubsection{Equivalences.}

Two efficient representatives ${\cal E} = (\Gamma, \Psi, h_{\Gamma})$ and 
 ${\cal E '} = (\Gamma ', \Psi ', h '_{\Gamma '})$ are {\em equivalent} if there exists a simplicial homeomorphism $ C: \Gamma \rightarrow \Gamma '$ and a homeomorphism $g: S_X \rightarrow S_X$, isotopic to the identity so that: \\ (Eq1) $ \Psi ' \circ C = C \circ \Psi$ and, \\
 (Eq2)  $ h' _{\Gamma '} \circ C = g \circ h_{\Gamma}$.\\
The homeomorphism $C$ induces a relabelling of the edges and (Eq1) states that after this relabelling the two maps are identical. The second condition 
requires that the relabelling respects the cyclic ordering, and that after the relabelling the two embeddings are isotopic on the surface.

In the previous notion of equivalence if  $g$ is not required to be isotopic to the identity,  we obtain a notion of equivalence for the pairs   $ (\Gamma, \Psi) $ and 
 $ (\Gamma ', \Psi ')$ that is called {\em combinatorial equivalence}.

The mapping class group $Mod(S_X)$ acts on the set of efficient representatives as follows:\\
if ${\cal E} = (\Gamma, \Psi, h_{\Gamma})$ is an efficient representative of $[f]$ and $[g ]\in Mod(S_X)$ then we denote  
$g^{*}{\cal E} = (\Gamma, \Psi, g\circ h_{\Gamma})$, it is an efficient representative of $[g\circ f\circ g^{-1}]$.

\subsection{Geometrical aspects.}
The goal of this section is to construct a flat surface together with a pair of measured foliations from the combinatorial data contained in an efficient representative 
${\cal E} = (\Gamma, \Psi, h_{\Gamma})$ of $[f]$. This construction provides all the geometric invariants of the pseudo-Anosov class $[f]$: dilatation factor, stable and unstable invariant measured foliations.

\subsubsection{Basic gluing operations.}
The very first metric properties comes from the incidence matrix 
$M(\Gamma, \Psi)$. From  property (G4) and Perron-Frobenius theorem, $M$ has a unique real maximal eigenvalue $\lambda > 1$, together with an eigenvector $L$ and another eigenvector $W$ for the transpose matrix $^tM$. These vectors are unique up to scale, $W$ has positive entries, the entries of $L$ are nonnegative and the entries of $L$ labelled by $H$ are positive.\\
Fix $L$ and $W$ in their projective classes. The edge $e$ of $\Gamma$ has an orientation, say from a vertex $v = i(e)$ to $v' = t(e)$. We associate to $e$ a pair of non negative numbers 
$(L(e), W(e))$ and a bifoliated rectangle $R(e)$ of length $L(e)$ and width $W(e)$. 
The orientation of $e$ induces an orientation of $R(e)$, with two opposite sides
$\partial_v R(e)$ and  $\partial_{v'} R(e)$. 
Note that some rectangles are degenerate, ie have positive width and zero length. The corresponding edges are called {\em infinitesimal}.

The graph $h_{\Gamma} (\Gamma)$ is embedded in $S_X$ and, around each vertex $h_{\Gamma} (v)$, there is a cyclic ordering of the edges 
$h_{\Gamma} (e_i)$ incident at $h_{\Gamma} (v)$, this set of edges of $\Gamma$ is the star $St(v)$. A small disc neighbourhood $D(v)$ of $h_{\Gamma} (v)$ is considered as a planar chart. The rectangles $R(e_i)$ are disjointly embedded into $S_X$ under ${\hat h}_{\Gamma}$ so that 
${\hat h}_{\Gamma}(R(e_i)) \bigcap D(v)$ are cyclically ordered as the  
$h_{\Gamma} (e_i)$ for $e_i \in St(v)$.\\
 The next goal is to find an identification of the sides 
$\partial_v {\hat h}_{\Gamma}R(e_i),\/ e_i \in St(v)$ in the chart $D(v)$. But some metric informations are missing: which proportion of the side $\partial_v {\hat h}_{\Gamma}R(e_i)$ has to be identified with the adjacent sides?  This is the goal of the next paragraph.

\subsubsection{Train track construction: blow-up and glue.}

The idea is to construct a Williams-Thurston train tracks representative out of an efficient representative (see for instance Harer and Penner \cite{HaPe92} for details about train tracks). What is missing is the so called "switch condition" at each vertex and this forces us to change the graph structure as well as the map in order to impose a "2-sides" condition. This is achieved by {\em blowing up} each vertex of the graph $\Gamma$, keeping track of the gate structure given by the map $\Psi$. The details can be found in Bestvina and Handel \cite{BH95}, we review here the constructions and the results.\\
(1) Replace each vertex $h_{\Gamma} (v)$ by a small disc $\Delta (v)$ in each chart $D(v)$.\\
(2) Along the boundary $\partial \Delta (v)$ identify all the points 
$h_{\Gamma} (e_i) \bigcap \partial \Delta (v)$ when the corresponding edges
$e_i$ belong to the same gate (given by $\Psi$). If $g(v)$ is the number of distinct gates at $v$ we obtain $g(v)$ points $\{ y_1, ..., y_g\}$ along $\partial \Delta (v)$ that are cyclically ordered as the gates of $ (\Gamma, \Psi, h_{\Gamma})$ at $v$.\\
(3) Two points $y_i \neq y_j$ are connected by an edge $\epsilon_{i,j}$ within $\Delta (v)$ if there is an edge $e \in \Gamma$ and an iterate $\Psi^k (e)$ that crosses a turn 
$(e' _i, e'_j)$ where $e'_i$ belongs to the gate $g_i$ and $e'_j$ to the gate $g_j$.

These (local) operations define a new graph $\tau(\Gamma, \Psi)$, an embedding $h_{\tau}:\tau \rightarrow S_X$ and a new map $\widetilde{\Psi}: \tau \rightarrow \tau$ that are well defined. Let us express what are the conclusions when ${\cal E} = (\Gamma, \Psi, h_{\Gamma})$ is an efficient representative of a pseudo-Anosov homeomorphism.

(i) The graph $\tau(\Gamma, \Psi)$ is connected.\\
(ii) The subgraphs $h_{\tau}(\tau (\Gamma, \Psi) ) \bigcap\Delta (v)$ are either $g(v)$-gones, or $g(v)$-gone minus one side.\\
(iii)  The map $\widetilde{\Psi}: \tau \rightarrow \tau$ is well defined, it permutes the $\epsilon$ edges and each vertex of $\tau$ has exactly two sides: the $\epsilon$-edges inside $\Delta (v)$ and one gate outside $\Delta (v)$.\\
(iv) The incidence matrix for $(\tau, \widetilde{\Psi})$ is defined exactly as for any topological representative and has the following form:
$$\left(
\begin{array}{ccc}
E & F \\
0 & M(\Gamma, \Psi)
\end{array}
\right),$$
where the block $E$ corresponds to the $\epsilon$-edges and is thus a permutation matrix by (iii). Exactly as for the matrix $M(\Gamma, \Psi)$, this new matrix  $M(\tau, \tilde{\Psi})$ has the same largest eigenvalue $\lambda > 1$ and two corresponding eigenvectors  $(\widetilde{L}, \widetilde{W} )$.\\
(v) The graph $\tau(\Gamma, \Psi)$ together with the weight function $\widetilde{W}$ is a {\em measured train track}, ie it satisfies at each vertex a 
{\em switch condition}: each vertex has two sides  by (iii) and the sum of the weights on one side equals the sum of the weights on the other side.

By (iv) the newly introduced edges $\epsilon$ are infinitesimal: $\widetilde{L}(\epsilon_{i,j}) = 0$, which imply that the length structure has not changed between $\Gamma$ and $\tau$. The width structure $\widetilde{W}(\epsilon_{i,j}) > 0$ is exactly the information that was missing in the preliminary naive construction: it gives the exact position of the singularities inside the discs $D(v)$, with respect to the transverse measure given by $\widetilde{W}$.


\begin{figure}[htb]
\centerline {\includegraphics[scale=0.5]{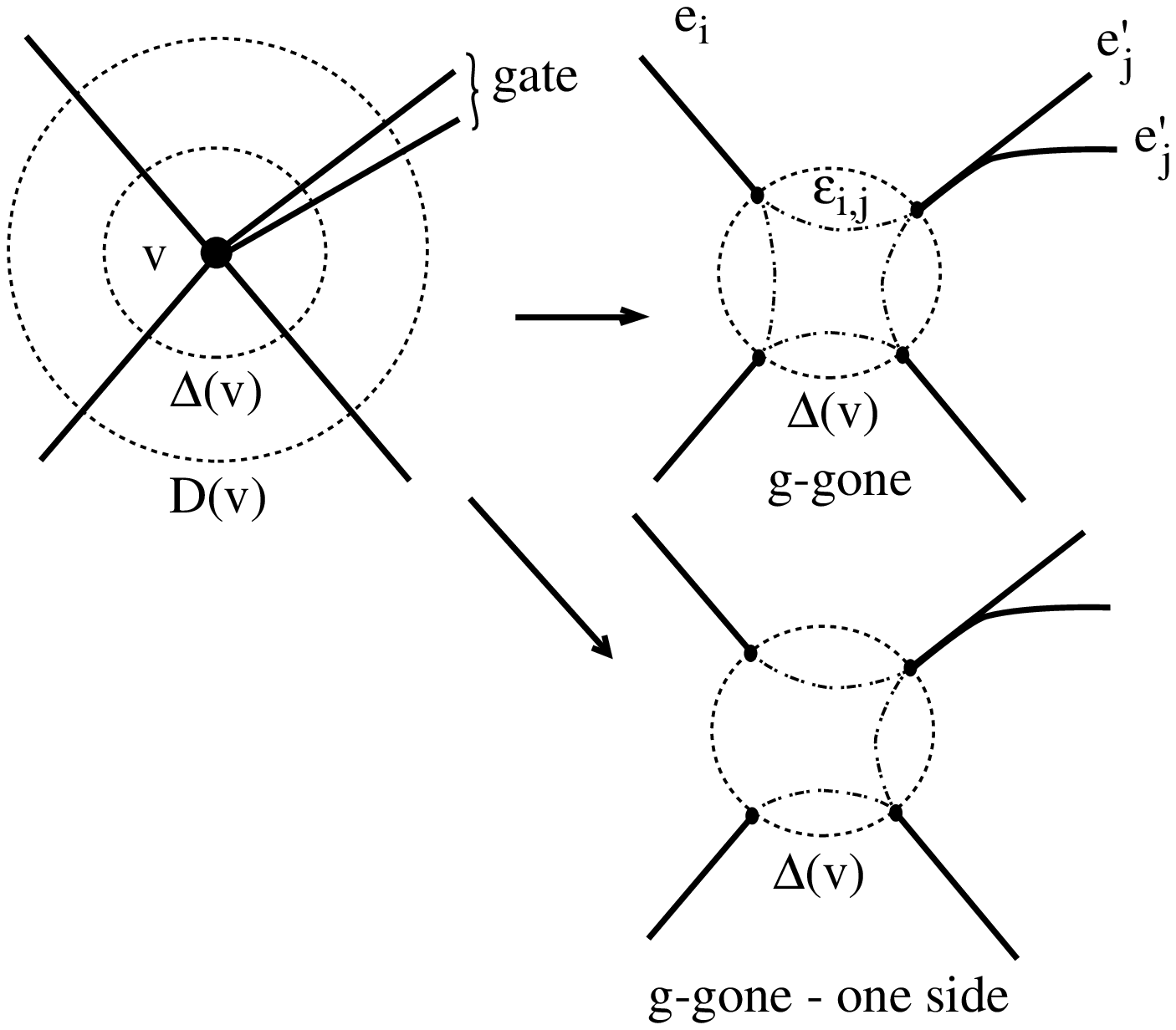}}
\label{fig:1}
\caption{ Train track construction at a vertex.}
\end{figure}

The construction is now essentially the classical "highway construction" of Thurston, see Harer and Penner \cite{HaPe92}, obtained by gluing rectangles with specific width satisfying switch conditions. The difference is that our rectangles have both width and length. For the rectangles with non zero length the identification is just the obvious one. For the rectangles of length zero, coming from the infinitesimal edges, the identification is done in two steps. First we fix for these (infinitesimal) rectangles an arbitrary small length $\delta$ and we construct a surface $\mathcal{R}^{[{\cal E}]}_{\delta}$ together with a pair of measured foliations. Then we let $\delta$ goes to zero and we obtain a surface
$\mathcal{R}^{[{\cal E}]}$ with a pair of measured foliations (see figures 1, 2, 3).


\begin{figure}[htb]
\centerline {\includegraphics[height=80mm]{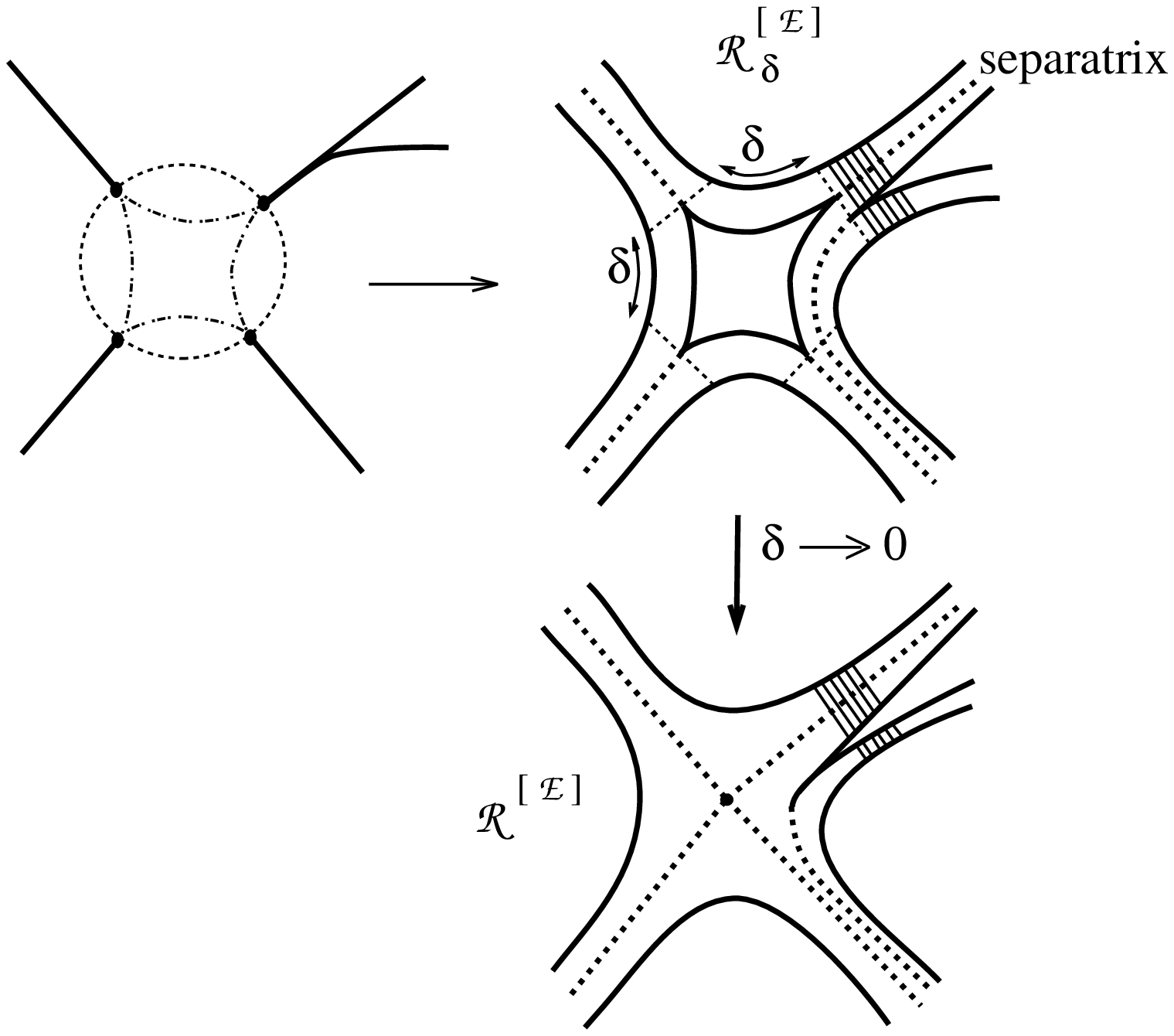}}
\label{fig:2}
\caption{ Foliated surface from train track}
\end{figure}

The surface $\mathcal{R}^{[{\cal E}]}$ has $n$ marked points, one boundary component and is homeomorphic to $Clos( S - D_{x_0} )$. The map 
$\Psi: \Gamma \rightarrow \Gamma$  (or  $\widetilde{\Psi}: \tau \rightarrow \tau$ ) induces a map 
$f_{\cal E}: \mathcal{R}^{[{\cal E}]}  \rightarrow \mathcal{R}^{[{\cal E}]}$
that is a homeomorphism in the neighbourhood of any interior point.

\begin{figure}[htb]
\centerline {\includegraphics[scale=0.5]{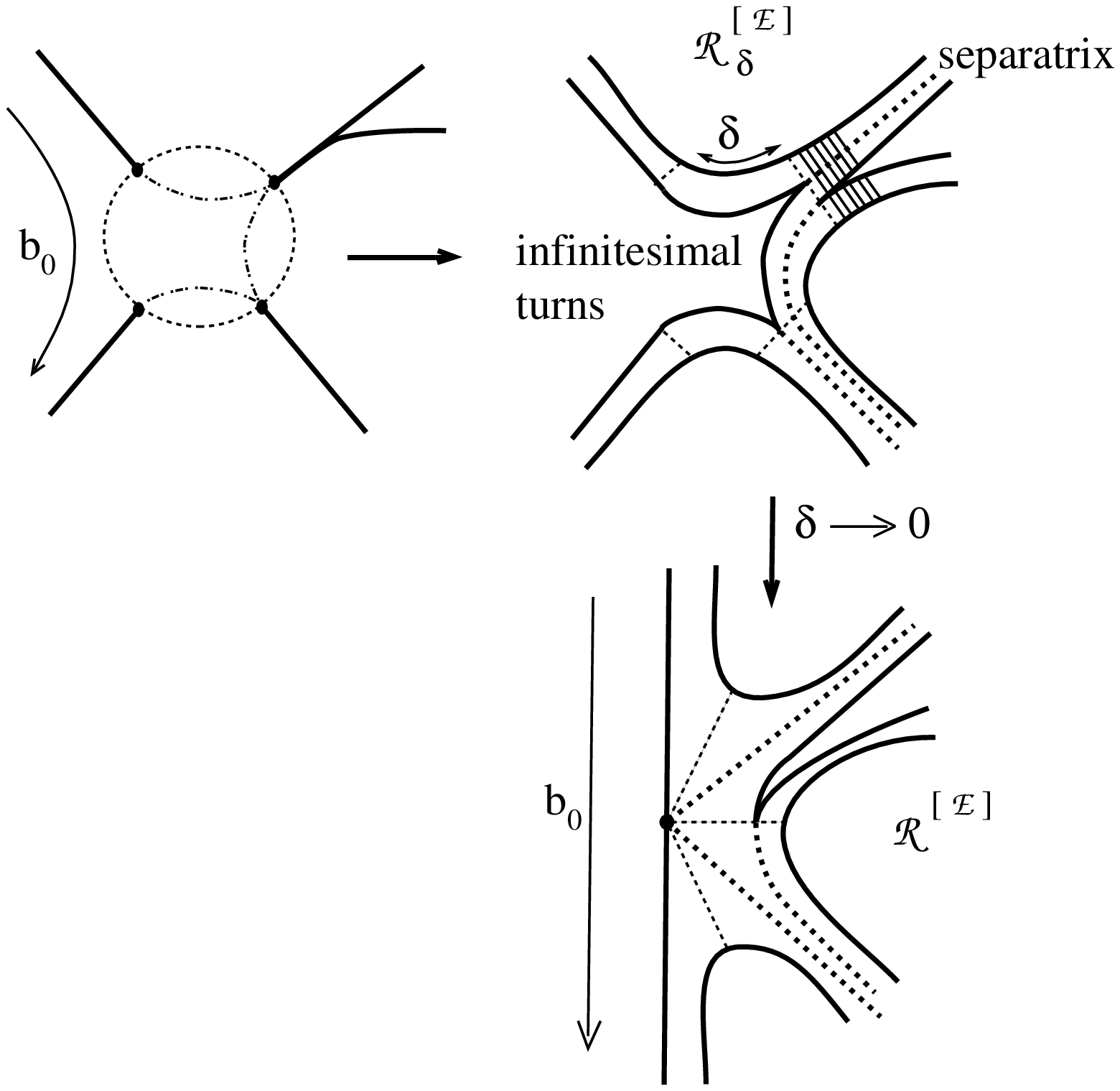}}
\label{fig:3}
\caption{Infinitesimal turns.}
\end{figure}

\subsubsection{ Periodic points on the boundary, Nielsen paths.}

 By the gluing construction, each illegal turn $ t = (a,b)$ of 
$(\Gamma, \Psi)$  (or of $(\tau, \tilde{\Psi})$ ) gives rise to a cusp $C_t$ of 
$\mathcal{R}^{[{\cal E}]}$ because of the condition (G5) and part (2) of the train track construction.
The map 
$f_{\cal E}: \mathcal{R}^{[{\cal E}]}  \rightarrow \mathcal{R}^{[{\cal E}]}$
 maps the boundary curve to itself, up to homotopy, is uniformly expanding along this curve and thus has periodic points on the boundary $\partial \mathcal{R}^{[{\cal E}]}$ with the following features:\\
- The periodic points alternate along $\partial \mathcal{R}^{[{\cal E}]}$, with the cusps $C_t$.\\
- The path between two consecutive periodic points, passing through a cusp is a periodic {\em Nielsen path}, see Jiang \cite{j83}.\\
The surface $ \mathcal{R}^{[{\cal E}]}$ has a flat metric given by the length and width structure of the constitutive rectangles. Similarly the graph $\Gamma$ together with the vector $L$ is a metric graph. Furthermore the rectangle partition of $ \mathcal{R}^{[{\cal E}]}$ is, by construction, a
{\em Markov partition} for the map $f_{\cal E}$ (see for instance Fathi {\em et al.} \cite{flp79}) from which a symbolic dynamics is well defined (see for instance Alseda {\em et al.} \cite{ALM93}, Shub \cite{sh78}).
This symbolic dynamics, together with the metric structure  on 
$ \mathcal{R}^{[{\cal E}]}$ (or on $\Gamma$) enables to give the exact metric location of the periodic points on $\partial \mathcal{R}^{[{\cal E}]}$ and thus of the Nielsen paths between them.

Let $t =(a,b)$ be an illegal turn of $(\Gamma, \Psi)$ and $C_t$ the corresponding cusp on $\partial \mathcal{R}^{[{\cal E}]}$. From the previous discussion there are 2 periodic points $(a_t , b_t)$ on 
$\partial \mathcal{R}^{[{\cal E}]}$ adjacent to $C_t$. Let $A_t = [C_t , a_t]$ and 
$B_t = [C_t , b_t]$ be the paths from $C_t$ to $a_t$ (resp. $b_t$) along 
$\partial \mathcal{R}^{[{\cal E}]}$. These paths have the same length:  $L(A_t) = L(B_t)$. The path $ N_t = \overline{A_t} B_t$ is the Nielsen path between $a_t$ and $b_t$ passing through $C_t$, and 
$\bigcup_t N_t=\partial \mathcal{R}^{[{\cal E}]}$  (see figure 4).

\begin{figure}[htb]
\centerline {\includegraphics[scale=0.4,angle=90]{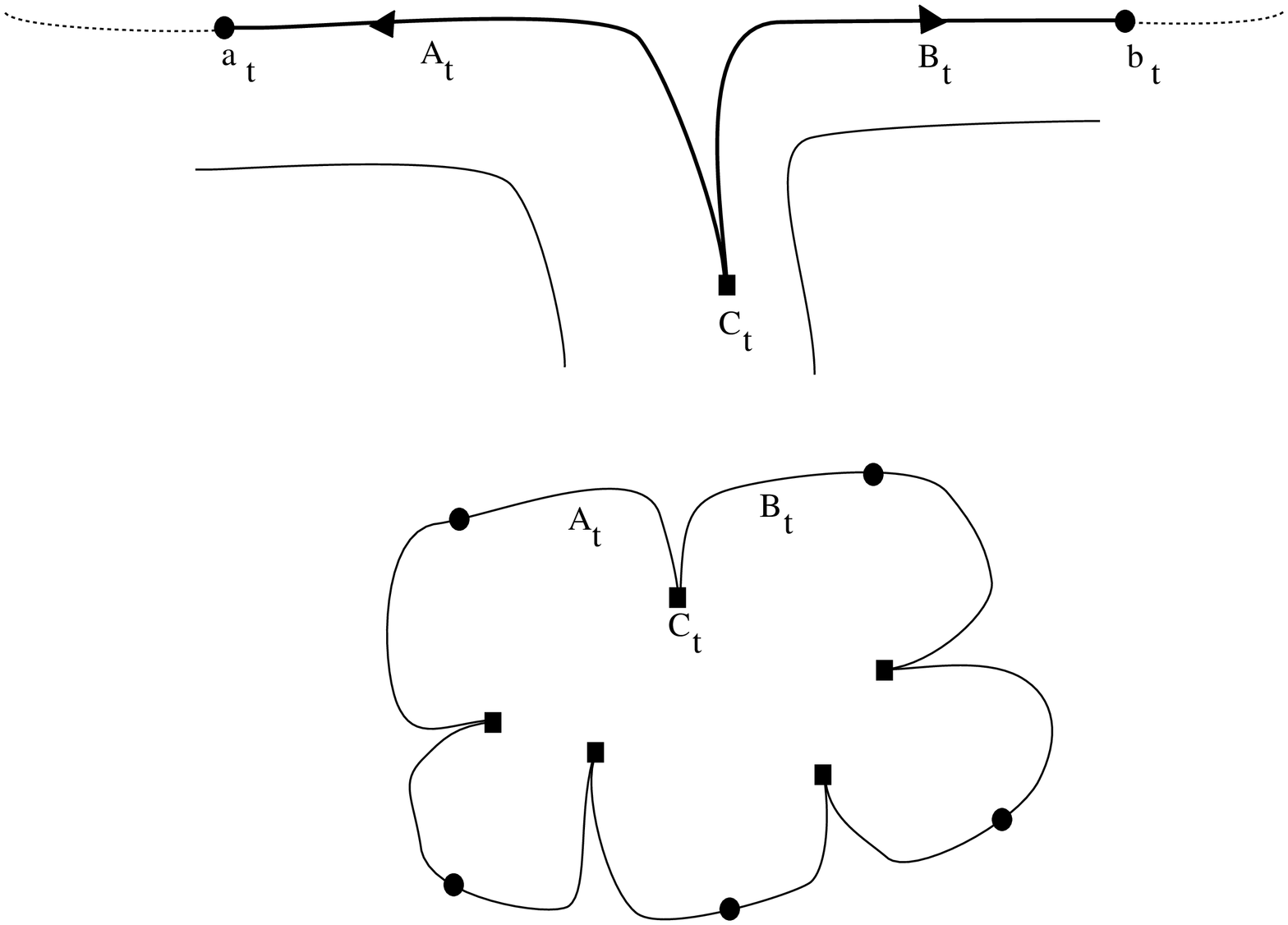}}
\label{fig:4}
\caption{ The Nielsen paths and the boundary .}
\end{figure}

If we identify isometrically $A_t$ with $B_t$ for each illegal turn $t$, then all the periodic points along $\partial \mathcal{R}^{[{\cal E}]}$ are identified to a single point that we identify with $x_0$, the quotient surface is identified with $S$, ie there is a map 
$ \mathcal{IN}:  \mathcal{R}^{[{\cal E}]} \rightarrow S $. 
The pair of measured foliations on $ \mathcal{R}^{[{\cal E}]}$ induces a pair 
$\{ (\mathcal{F}^s,\mu^s); ( \mathcal{F}^u, \mu^u) \}$ on $S$ and the identified arcs give rise to segments of length $L(A_t) = L(B_t)$ along unstable leaves starting from the marked point $x_0$. There is thus a mapping from the set of illegal turns of the efficient representative ${\cal E}$ to a subset of the unstable leaves issued from the marked point $x_0$ (the separatrices). 
The map $f_{\cal E}$ passes
 to the quotient $ \mathcal{IN}$ and the quotient map $f$ is a homeomorphism. 
The resulting homeomorphism is thus the pseudo-Anosov homeomorphism $f$: it fixes $x_0$, leaves invariant the pair of measured foliations $\{ (\mathcal{F}^s,\mu^s); ( \mathcal{F}^u, \mu^u) \}$ with dilatation factor $\lambda > 1$.

The previous construction has been given from a specific efficient representative 
$(\Gamma, \Psi, h_{\Gamma} )$ that is a combinatorial object. It allows to reconstruct $(S, f)$ and its pair of invariant measured foliations 
$\{ (\mathcal{F}^s,\mu^s); ( \mathcal{F}^u, \mu^u) \}$, recall that the dilatation factor $\lambda$ has been computed during the process as the largest eigenvalue of the matrix $M(\Gamma, \Psi)$. It also gives an injective map from the set of illegal turns to a subset of the unstable separatrices at the fixed marked point $x_0$.
The possible missing separatrices are obtained as follows. In the construction of the train track $\tau (\Gamma, \Psi)$, the set of infinitesimal edges 
$\epsilon$ in a disc $D(v)$ can only have two shapes: a $g$-gone or a $g$-gone minus one side, by (ii) of paragraph 2.2.2. In the first case the corresponding singularity is interior to 
$ \mathcal{R}^{[{\cal E}]}$. 
In the second case, a turn between two infinitesimal edges is called an {\em infinitesimal turn}, it is crossed by a boundary loop $b_i$, $i = 0, 1, ..., n$.
If the infinitesimal turn is crossed by a loop $b_i$ then it defines a separatrix at the marked point $x_i$. In particular those infinitesimal turns that are crossed by $b_0$ define separatrices at $x_0$ (see figure 3).
We just described a one-to-one correspondence between the unstable separatrices at $x_0$ and some turns of $(\tau, \widetilde{\Psi})$ that are well defined from the efficient representative ${\cal{E}} = (\Gamma, \Psi, h_{\Gamma} )$.

\begin{lem}
Let ${\cal{E}} = (\Gamma, \Psi, h_{\Gamma} )$ be an efficient representative of a pseudo-Anosov element $[f]$ and $(\tau, \widetilde{\Psi}) ({\cal{E}})$ the associated train track map. 
There is a well defined bijective map
$u_{\cal{E}}: T^0_{\cal{E}} \mapsto \{u_1, ..., u_m\}$, where 
$ \{u_1, ..., u_m\}$ is the set of unstable separatrices at the fixed marked point $x_0$ and $T^0_{\cal{E}}$ is the union of the illegal turns and the infinitesimal turns that are crossed by $b_0$ in $(\tau, \widetilde{\Psi}) ({\cal{E}})$.
There is another well defined map 
$\mathcal{L}_{\cal{E}}: T^0_{\cal{E}} \mapsto \mathbb{R}^m$ given by 
$\mathcal{L}_{\cal{E}} (t) = \frac{1}{2} length (N_t)$ if $t \in  T^0_{\cal{E}}$ is an illegal turn and $ \mathcal{L}_{\cal{E}} (t) = 0 $ otherwise, where $N_t$ is the Nielsen path between two consecutive boundary periodic points associated with the illegal turn $t$.
$\square$
\end{lem}

\section{Cut lengths, single cut representatives.}

The goal of this section is to prove a converse of the previous lemma, ie starting from the geometric invariants of the pseudo-Anosov homeomorphism $f$ together with a length vector as above and satisfying some additional properties, we  construct an efficient representative of $[f]$.
The construction is similar to the one in Fathi {\em et al.} \cite{flp79} used to prove the existence of a Markov partition for pseudo-Anosov homeomorphisms. In addition it proves the existence of efficient representatives without appealing to the train track algorithm.

\subsection{Admissible cut lengths.}

First we fix once and for all the invariant measured foliations in their respective projective classes, in other words the measures $\mu_s$ and $\mu_u$ are fixed.
Recall that the pseudo-Anosov $f$ fixes the marked point $x_0$ and induces a permutation $\sigma: \{1, ..., m\} \to  \{1, ..., m\}$ of the unstable separatrix indexes at $x_0$. Let $l_s = (l_1, ..., l_m)$ be a non zero vector 
in $\mathbb{R}_{+}^{m}$, called a {\em cut length} vector.
The unstable separatrices $\{u_1, ..., u_m\}$ are cyclically indexed and we consider the initial segment $[x_0, u_i (l_i)]_{u_i}$ along $u_i$ of $\mu_s$-measure $l_i$
for $ 1 \leqslant i \leqslant m $, called a {\em  cut line} along $u_i$.
We denote by $\mathcal{L} (l_1, ..., l_m)$ the union of the cut lines.
There is a one-to-one correspondence between the cut length vector $(l_1, ..., l_m) \in \mathbb{R}_{+}^{m}$ and the cut lines $\mathcal{L} (l_1, ..., l_m) \subset S$.


\begin{figure}[htbp]
\centerline {\includegraphics[scale=0.5]{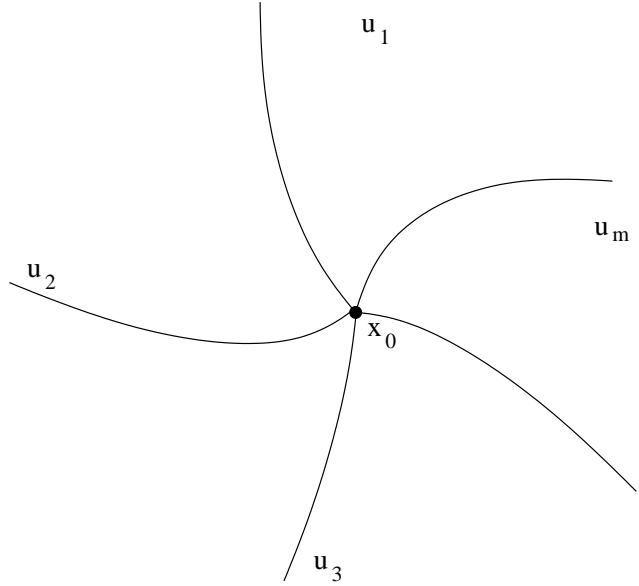}}
\label{fig:5}
\caption{ The unstable separatrices at the point $x_0$.}
\end{figure}

The surface $S$ together with 
$(\mathcal{F}^s, \mu^s) ; (\mathcal{F}^u, \mu^u)$ and the cut lines $\mathcal{L} (l_1, ..., l_m)$ defines two types of special points:\\
- The {\em true singularities} are the singularities of the foliations $\mathcal{F}^s$ and $\mathcal{F}^u $.\\
- The {\em  artificial singularities} are the end points of the cut lines (with non zero cut length).\\
Consider the following set of segments: $\mathcal{W} (l_1, ..., l_m)$ (resp.
$\mathcal{W}^{true} (l_1, ..., l_m)$), that are the union of the stable separatrices starting from all the singularities (resp. true singularities) up to the first intersection with $\mathcal{L} (l_1, ..., l_m)$.  This set is a finite union of compact segments by minimality of $\mathcal{F}^s$, see Fathi {\em et al.} \cite{flp79}.

\begin{defn}
A cut length vector $( l_1, ..., l_m)$ is {\em admissible} (for the pseudo Anosov homeomorphism $f$ with dilatation factor $\lambda$) if the following conditions are satisfied:\\
(i) There is one orbit $\mathcal{O}$ of $\sigma$ such that: $l_i \neq 0$ if and only if $i \in \mathcal{O}$.\\
 There exists  $i_0 \in \mathcal{O}$ such that:\\
(ii) $l_{\sigma(i)} =  \lambda l_i$ if $i \neq i_0$.\\
(iii) $u_{i_0} (l_{i_0}) \in \mathcal{W}^{true}(l_1, ..., l_m)$.

\end{defn}

\begin{lem}
Admissible cut length vectors exist.
\end{lem}

\begin{proof} Let $\mathcal{O}$ be one orbit of $\sigma$. Define
$(d_1, ..., d_m)$ such that $d_i = 1$ if $i \in \mathcal{O}$ and $d_i = 0$ otherwise (of course this is non admissible). The set of true singularities is non empty so at least one segment $[x_0, u_{i_0}(d_{i_0})]_{u_{i_0}}$ contains a point $y_0 \in \mathcal{W}^{true}(d_1, ..., d_m)$. 
Define then $l_i = 0 $ if $i\notin \mathcal{O}$, 
$l_{i_0} = l = \mu^s ( [x_0, y_0]_{u_{i_0}})$ and 
$ l_{\sigma^{ - p} (i_0)} = \lambda ^{ - p}  l$ for $p = 1,..., q-1$ where $q$ is the order of $\sigma$.
The first two conditions are satisfied by construction. Next we observe that 
$l_i \leq d_i$ for all $i$, hence
$\mathcal{L}(l_1, ..., l_m) \subset \mathcal{L}(d_1, ...,d_m)$, this implies that 
$\mathcal{W}^{true}(d_1, ..., d_m) \subset \mathcal{W}^{true}(l_1, ...,l_m)$
and therefore $y_0 \in \mathcal{W}^{true}(l_1, ...,l_m)$.
\end{proof}

\subsection{ Single cut representatives.}

The goal of this section is to prove the existence of a canonical subclass among efficient representatives of $[f]$ satisfying some additional properties. This subclass is called the {\em single cut representatives} and the study of their  properties is the main focus of this paper. 

\subsubsection{ Construction of an embedded graph.}

Let $(l_1, ..., l_m)$ be an admissible cut length vector and let 
$\mathcal{L}(l_1, ..., l_m)$,  $\mathcal{W}(l_1, ..., l_m)$ be the associated segments defined above.

{\em Claim:}    $S - \left(\mathcal{L}(l_1, ..., l_m) \cup \mathcal{W}(l_1, ..., l_m)\right)$ is a disjoint union of bifoliated rectangles.

By definition, all the singularities of the foliations 
$\mathcal{F}^s, \mathcal{F}^u$ belongs to 
$\mathcal{L}(l_1, ..., l_m) \cup \mathcal{W}(l_1, ..., l_m)$. So each connected component of $S - \left(\mathcal{L} \cup \mathcal{W}\right)$ is bifoliated by restriction of 
$\mathcal{F}^s, \mathcal{F}^u$ and has no singularity in its interior. Therefore it is a rectangle $R$. The closure $\overline{R}$ of each component $R$ has two sides that are segments in  $\mathcal{L}(l_1, ..., l_m)$ and two sides in 
$\mathcal{W}(l_1, ..., l_m)$. Each $\overline{R}$ is embedded in $S$.
$\square$\\
We define a graph $\Gamma$ and an embedding $h_{\Gamma}$ as follows. 
There are 3 types of edges:\\
(1) One edge for each rectangle $\overline{R}$ of $S - \left(\mathcal{L} \cup \mathcal{W}\right)$,  they are called {\em H-edges}.\\
(2) One edge for each segment of $\mathcal{W}$ issued from a distinguished point $x_i, i\geqslant 1$, they are called {\em peripheral}.\\
(3) One edge for each pair of segments of $\mathcal{W}$ issued from an artificial singularity that does not contain a true singularity and are eventually mapped under $f$ to a segment of $\mathcal{W}$ issued from a distinguished point. These edges are called {\em preperipheral} (see figure 6).






\begin{figure}[htb]
\centerline {\includegraphics[scale=0.5]{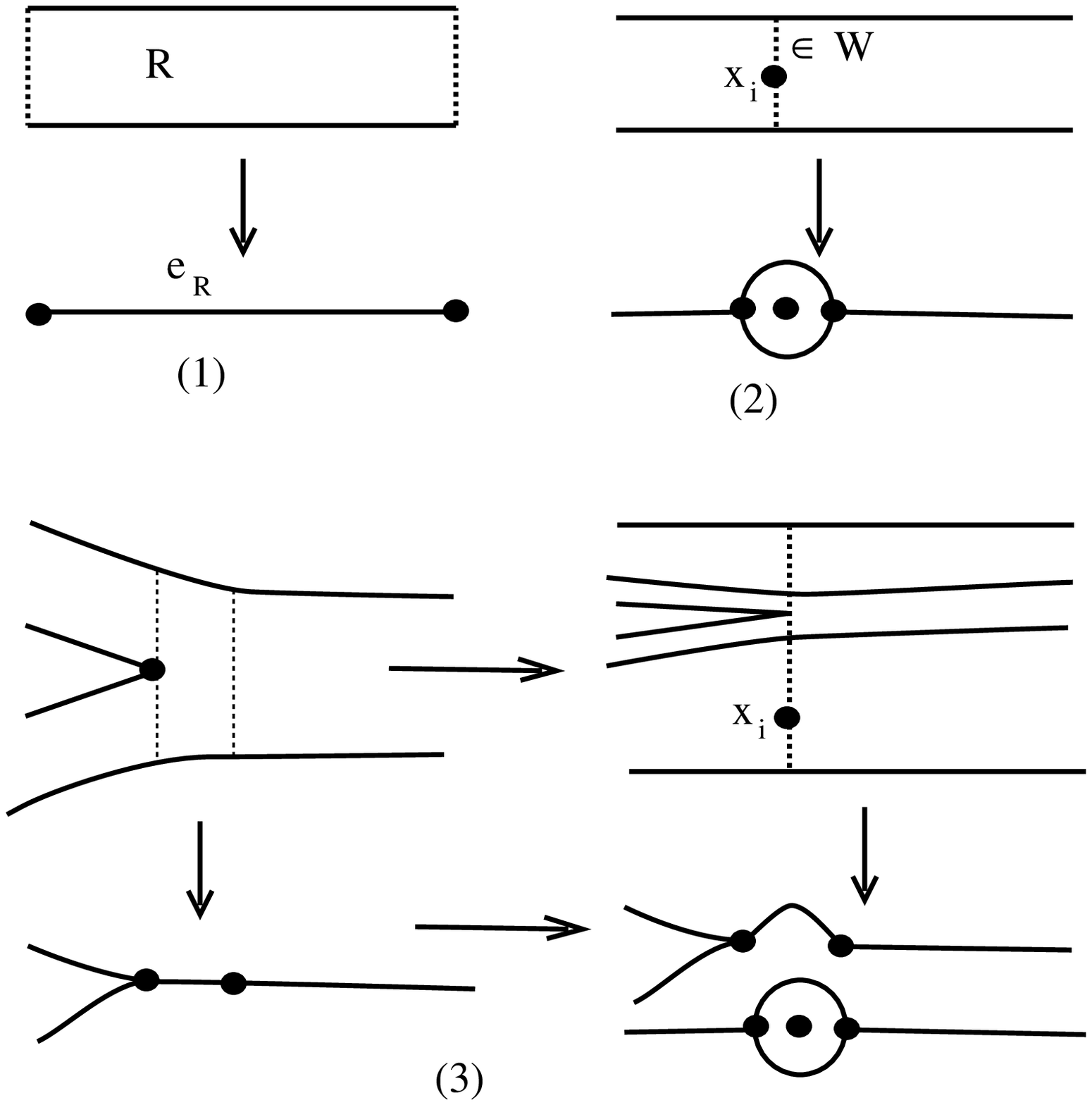}}
\label{fig:6}
\caption{   Construction of a graph.}
\end{figure}

We consider the segments of $\mathcal{W}$ issued from the distinguished points $x_i, i\geqslant 1$ in (2) and (3) as degenerate rectangles (with zero length).
The edges just defined are identified at their extreme points if the corresponding rectangles share a common $\mathcal{W}$ segment. These identifications  define the vertices of the graph denoted $\Gamma$.
Each rectangle is embedded in $S$ and the identification of the extreme points respects the embedding, in other words there is an embedding 
$h_{\Gamma}: \Gamma \rightarrow S_X$ that is well defined up to isotopy.

In addition there is a retraction $\rho:  S_X \rightarrow \Gamma$ obtained by contracting each segment of the stable foliation inside each rectangle to a point on $\Gamma$. This retraction is a homotopy inverse of $h_{\Gamma}$ and thus $(\Gamma, h_{\Gamma})$ is an embedded graph.

{\em Remark.} In the case with only one distinguished point, the previous construction is much simpler: no peripheral and preperipheral edges are necessary.

\subsubsection{ Construction of a map.}

The segments of $\mathcal{L}$ belong to one $f$-orbit of unstable separatrices at  $x_0$ and $\mathcal{W}$ is the union of all the segments of stable separatrices starting at the singularities. This implies $f^{-1} (\mathcal{L}) \subset \mathcal{L}$ and $f (\mathcal{W}) \subset \mathcal{W}$. 
Thus the image under $f$ of each rectangle $R$ of 
$S -  \left(\mathcal{L}\bigcup \mathcal{W}\right)$ is a rectangle in  
$S - \mathcal{L}$, starting at a $\mathcal{W}$-component and ending at another $\mathcal{W}$-component.
This is a Markov partition property, see Fathi {\em et al.} \cite{flp79}.
Let $R$ be such a rectangle and $e_R$ the corresponding edge in $\Gamma$.
Any maximal unstable segment $U$ in $R$, is mapped to an unstable segment 
$f(U)$ that crosses a finite collection of rectangles and $\mathcal{W}$-components issued from the marked points. This sequence of rectangles is independent of the choice made for the segment $U$ in $R$.
From the definition of $\Gamma$, the sequence of rectangles and 
$\mathcal{W}$-components that are crossed by $f(U)$ defines an edge path $p$
in $\Gamma$.
A special case is when $f(U)$ start or end at a $\mathcal{W}$-components $w$ issued from a marked point or at an artificial singularity giving rise to a peripheral or preperipheral edge in $\Gamma$.
In this case we make the choice that the edge path $p$ does not start or end with the corresponding peripheral or preperipheral edge. This ensures condition (G5). This defines the image of the $H$-edges. 
A peripheral edge $e_w$ is associated to a stable segment $w$ issued from a distinguished point. Its image is the peripheral edge corresponding to the segment containing $f(w)$. The image of a preperipheral edge associated to a stable segment $w'$ is either the preperipheral or peripheral edge associated to 
$f(w')$.
All together this defines a map $\Psi: \Gamma \rightarrow \Gamma$.

\begin{proposition}
The triple $\mathcal{E} = (\Gamma, \Psi, h_{\Gamma})$ constructed above 
is an efficient representative of $[f]$. It satisfies the following additional properties:\\
(SC1)  There is exactly one illegal turn of order one, between two $H$-edges at a vertex with more than two gates.\\
(SC2) Every vertex that is crossed by a loop $b_i , i \geqslant 1$ around a distinguished point has exactly 3 gates.\\
An efficient representative satisfying these additional properties is called a
{\em SC-repre\-sen\-ta\-ti\-ve}. If the leaf segment associated to the illegal turn of order one is $u_{i_0}$ we  add: {\em relative to $u_{i_0}$}.
\end{proposition}
\begin{proof}  We already proved that $(\Gamma, h_{\Gamma})$ is an embedded graph. The map $\Psi$  is continuous since connected components of $\mathcal{W}$ are mapped to connected components of $\mathcal{W}$ under $f$ so vertices are mapped to vertices under $\Psi$ and continuity comes from the continuity of $f$ on the rectangles.
Next we check that $\Psi$ represents $[f]$. To this end we already defined a retraction $\rho: S_X \rightarrow \Gamma$ by contracting the rectangles along the stable segments. The definition of $\Psi$ via the action of $f$ on the unstable segments in each rectangle implies that $\rho  \circ f$ is homotopic to
$\Psi \circ \rho$. The retraction $\rho$ being an homotopy inverse of $h_{\Gamma}$, we obtain that: 
$ f \circ h_{\Gamma} \simeq h_{\Gamma} \circ \Psi$.
Let us now prove the other properties:\\
(G1) The graph $\Gamma$ has no valency one vertices, because every rectangle is glued to rectangles at both of its $\mathcal{W}$ sides. The other properties of a topological representative were proved above.\\
(G2) The peripheral and preperipheral edges are mapped by $\Psi$ to peripheral or preperipheral edges,  in particular words have combinatorial length one. No cancellation can occur when iterating these edges. The iterates $\Psi^k (e)$ of an $H$-edge is defined by iterating an unstable segment. The corresponding edge path cannot backtrack by invariance of the unstable foliation $\mathcal{F}^u$ under $f^k$.\\
(G3) Comes from the definition of peripheral edges.\\
(G4) The matrix $M_H$ is irreducible. Indeed any sufficiently long unstable segment crosses every rectangle by minimality of $\mathcal{F}^u$ and for each maximal unstable segment in any rectangle a large enough iterate under $f$ is as long as we wish and therefore crosses every rectangle.\\
(G5) Comes from our choice of the image of $H$-edges.\\
(SC1) Illegal turns only occurs between $H$-edges by construction. Two $H$-edges define an illegal turn when the corresponding rectangles are on both sides of a cut line along $u_i$ and are glued together by a corner. This corner is thus the end point of a cut line $u_i(l_i)$. When the leaf satisfies 
$l_{\sigma (i)} = \lambda  l_i$ then the image under $f$ of the corner $u_i(l_i)$
is exactly the corner $u_{\sigma (i)}(l_{\sigma (i)})$. The image under $\Psi$ of the turn  corresponding to $u_i(l_i)$ is thus the turn corresponding to 
$u_{\sigma (i)}(l_{\sigma (i)})$ and this turn is of order greater than one. 
There is only one leaf  $u_{i_0}$ so that $l_{\sigma (i_0)} < \lambda  l_{i_0}$ by the admissibility property. The corresponding turn is illegal of order one. The fact that the vertex at which this specific  illegal turn arises has more than two gates comes from property (iii) of definition  3.1.\\
(SC2) Comes from the definition of the peripheral edges and their images. \end{proof}

\begin{lem}
Let $(l_1, ..., l_m)$ be an admissible cut length vector and 
$\mathcal{E} = (\Gamma, \Psi, h_{\Gamma} )$ the SC-representative obtained by the above construction. Then 
$\mathcal{L}_{\mathcal{E}}( T^{0}_{\mathcal{E}} ) =  (l_1, ..., l_m) \in \mathbb{R}_{+}^{m}$, where  
$\mathcal{L}_{\mathcal{E}}$ is the map defined in lemma 2.2.
\end{lem}

\begin{proof} The stable leaf segments from $x_0$ up to the first intersection with $\mathcal{L}$ belong to $\mathcal{W}$. These segments are permuted under $f$ in cycles of period $q$ (the order of the permutation $\sigma$).
The definition of the map $\Psi$ from $f$ implies that these segments are retracted to $q$ points on $\Gamma$ under the retraction $\rho$, and are permuted in a cycle of period $q$ under $\Psi$.
From the train track construction of paragraph 2.2.2, these periodic points  belong to the boundary component of the surface $\mathcal{R}^{\mathcal{E}}$ that is identified with $\mathcal{L}$. Therefore these points are exactly the boundary periodic points that are Nielsen equivalent.
The length of the Nielsen path, measured with  $\mu^s$, between two consecutive such points and crossing the illegal turn $t_i$ is by construction $2 l_i$, where  the index of the turn and the index of the corresponding unstable separatrix are identified.
It remains  to check that the $\mu^s$ length coincides with the eigenvector length given by Perron-Frobenius's theorem. Let $R_j$ be a rectangle of the partition given by 
$S - \left(\mathcal{L} \bigcup \mathcal{W}\right)$ and let $L_j$ be its $\mu^s$ length which is well defined since all the maximal unstable segments in $R_j$ 
have the same $\mu^s$ length. The image under $f$ of any maximal unstable segment in $R_j$ has $\mu^s$ length $\lambda  L_j$. This length is also the sum of the $\mu^s$ length of the rectangles crossed by this unstable segment (no back-tracks). Therefore the vector $L$ is an eigenvector of the incidence matrix $M(\Gamma, \Psi)$ for the eigenvalue $\lambda$.
\end{proof}
The following is a converse of  lemma 3.4.

\begin{proposition}
Let  $\mathcal{E} = (\Gamma, \Psi, h_{\Gamma} )$ be a  SC-representative of $[f]$. Then 
$\mathcal{L}_{\mathcal{E}}( T^{0}_{\mathcal{E}} ) =  (l_1, ..., l_m) \in \mathbb{R}_{+}^{m}$ is an admissible cut length vector.
\end{proposition}

\begin{proof}
If $t_i$ is an illegal turn of order greater than one then it is mapped exactly to another turn
$t_j$. The turn being illegal means that some of its iterate is degenerate, ie the beginning of the two edges are mapped to a non trivial common edge path. From paragraph 2.2.3, for each illegal turn $t_i$ there is a pair of periodic points on the boundary of the surface $\mathcal{R}^{\mathcal{E}}$ that are Nielsen equivalent for some iterate. The corresponding periodic Nielsen path has length $l(N_i)$. If the turn is illegal of order greater than one then the image of $N_i$ is a Nielsen path between the image periodic points, it crosses the illegal turn $t_j$
and has length $ l (N_j) = \lambda  l(N_i)$. The construction of the surface
$\mathcal{R}^{\mathcal{E}}$  and its pair of measured foliations associates
to an illegal turn an unstable separatrix by the map $u_{\mathcal{E}}$ of lemma 2.2. Therefore if $u_i$ is associated to $t_i$ 
then $u_{\sigma(i)}$ is associated to the image turn $t_j$ and the above equality reads: $ l_{\sigma(i)} = \lambda  l_i$ for all illegal turns in its orbit except the one of order one for which:
$ l_{\sigma(i_{0})} < \lambda  l_{i_{0}}$. Since there is only one illegal turn of order one then only one such orbit exists and it is associated to a single orbit of unstable separatrices.
The fact that the illegal turn $t_{i_{0}}$ of order one is located at a vertex with more than 3 gates implies, from the construction of the surface  $\mathcal{R}^{\mathcal{E}}$, that the cusp of the surface corresponding to $t_{i_{0}}$ is the extremity of a cut line and belongs to $\mathcal{W}^{true}$.
\end{proof}

Note that lemma 3.4 and proposition 3.5 imply the existence of a bijection between admissible cut lengths and SC-representatives.

\section{Operations on the set of SC-representatives.}

\subsection{SC-foldings.}

In this paragraph we define an operation on the set of SC-representatives, using the bijection between SC-representatives and admissible cut length vectors that was proved in the previous section. This operation is defined on cut lengths and in section 6 it will be translated in the combinatorial setting.
Let us start with some preliminary observations. Let $\mathcal{E}$ be a SC-representative with cut lengths $(l_1, ..., l_m)$. The {\em total length} (also called the volume in other contexts) of 
$\mathcal{E}$ is $l(\mathcal{E}) = \sum_{e\in Edge(\Gamma)} l(e)$.
For SC-representatives this length is also given by 
$l(\mathcal{E}) = \sum_{i = 1}^{ m} l_i$.\\
{\bf Observation 1:} If $\mathcal{E}$ is a SC-representative relative to the separatrix $u_{i_0}$ then:\\
$l(\mathcal{E}) = l_{i_0} ( 1 + \frac{1}{\lambda} + ...+  \frac{1}{\lambda ^{q-1}} )$, where $q$ is the period of the separatrix $u_{i_0}$ (the order of the permutation $\sigma$).\\
This obvious computation implies that SC-representatives relative to $u_i$ are parametrised by  the length $l_{i_0}$.\\
{\bf Observation 2:}  The image, under $f^q$, of any artificial singularity belongs to $\mathcal{W}^{true}$.\\
This is just the observation that an artificial singularity of a SC-representative is the end point of a cut line with positive cut length, together with: 
$u_{i_0} \in \mathcal{W}^{true}$ and 
$f(\mathcal{W}^{true}) \subset \mathcal{W}^{true}$.

\begin{defn}
Let $\mathcal{E}$ be a SC-representative relative to $u_{i_0}$ of cut length 
$(l_1, ..., l_m)$ and let $l >0$ be such that 
$u_{i_0} (l_{i_0} - l) \in \mathcal{W}^{true}$. The {\em SC-folding of length $l$} is given on cut lengths by the map 
$ (l_1, ..., l_m) \rightarrow (l'_1, ..., l'_m)$, where:\\
$l'_{i_0} = l_{i_0} - l$ and $l'_{\sigma^{-p}({i_0})} = l_{\sigma^{-p}({i_0})} - \frac{1}{\lambda^p} l $  for $p = 1, ..., q-1$ and\\
 $l'_j = l_j = 0$ if $j$ does not belongs the the $\sigma$-orbit of $i_0$.
\end{defn}

\begin{proposition}
If $(l'_1, ..., l'_m)$ is obtained from an admissible cut length $(l_1, ..., l_m)$ by a SC-folding of length $l>0$ then $(l'_1, ..., l'_m)$ is admissible.
We will say that the corresponding SC-representative $\mathcal{E '}$ is  obtained 
from $\mathcal{E}$ by a SC-folding of length $l$. 
Its total length is \\
$ l(\mathcal{E '}) = l (\mathcal{E}) - 
( 1 + \frac{1}{\lambda} + ...+  \frac{1}{\lambda ^{q-1}} )l$.
\end{proposition}

\begin{proof} By definition of $l$: 
$u_{i_0} (l_{i_0} - l) \in \mathcal{W}^{true}(l_1, .., l_m)$ and, since 
$ l'_i \leqslant l_i$ for all $i$, then  
$\mathcal{W}^{true}(l'_1, .., l'_m) \supset \mathcal{W}^{true}(l_1, .., l_m)$, proving admissibility. \end{proof}

\subsubsection{Elementary SC-foldings.}

First we prove that SC-foldings do exist for each admissible cut length.

\begin{lem}
Let $\mathcal{E}$ be a SC-representative relative to the leaf $u_{i_0}$ of cut lengths $(l_1, .., l_m)$. There is a point of  $\mathcal{W}^{true}(\mathcal{E})$ on the interior of the cut line 
$[ x_0, u_{i_0} (l_{i_0}) ]_{u_{i_0}}$.
\end{lem}

\begin{proof} The set of singularities is non empty and the number of separatrices is larger than one, so there is at least a cut line $i$ so that:
 $\mathcal{W}^{true} \cap [ x_0, u_{i} (l_{i}) ]_{u_{i}} \neq \emptyset$.
If the index $i \neq i_0$ then the $f$-image of the cut line  along $u_i$ is the cut line along $u_{\sigma (i)}$, and $f(\mathcal{W}^{true}) \subset \mathcal{W}^{true}$, so $u_{\sigma (i)}$ also contains a point of $\mathcal{W}^{true}$. Iterating this argument shows that the cut line along $u_{i_0}$ contains a point of $\mathcal{W}^{true}$.
\end{proof}
This lemma shows that  a given SC-representative $\mathcal{E}$ always admits a SC folding. Since the cut line 
 is totally ordered there is a first point $p_0$ of $\mathcal{W}^{true}$ along
$[ x_0, u_{i_0} (l_{i_0}) ]_{u_{i_0}}$ starting from $u_{i_0} (l_{i_0}) $ and we denote $l_0 = \mu^s [p_0,  u_{i_0} (l_{i_0}) ]_{u_{i_0}} > 0$.

\begin{defn} 
The SC-folding of length $l_0$ is called an {\em elementary SC-folding}.
\end{defn}

By definition, $l_0$ is the smallest possible length for which a SC-folding can be applied to $\mathcal{E}$. Observe  that $l_0$ is a $\mu^s$-length of a maximal unstable segment in $S - \left(\mathcal{L}\cup \mathcal{W}^{true}\right)$ therefore it is larger than the minimal edge length: $l_0 \geqslant min_{e \in Edge(\Gamma)} l(e) $, where the minimum is taken among edges of positive length, ie among $H$-edges.

\begin{proposition}
Let $\mathcal{E}$ be a SC-representative relative to $u_{i_0}$ and $l_0$ the length of the elementary SC-folding at $\mathcal{E}$. 
Then no point of $\mathcal{W}^{true}$ belong to the segments \\
$ \big{[}  u_{\sigma^{-p}(i_0)} (\frac{1}{\lambda^p}( l_{i_0} - l_0 ) ), 
u_{\sigma^{-p}(i_0)} (\frac{1}{\lambda^p}( l_{i_0}) ) \big{]}$ for
 $p = 1, ..., q-1$.
\end{proposition}

\begin{proof} This follows from the proof of lemma 4.3 and the definition of $l_0$. \end{proof}

\subsubsection{Action of $[f]$ on the set of SC-representatives. }

We defined in 2.1.3 a natural action of the mapping class group on the set of efficient representatives of $[f]$ given by 
$g^*\mathcal{E} = (\Gamma, \Psi, g\circ h_{\Gamma} )$, for  $[g] \in Mod(S_X)$. $g^*\mathcal{E}$  is an efficient representative of $ [g\circ f \circ g^{-1}]$. When $[g]$ belongs to the centraliser of $[f]$ then $g^*\mathcal{E}$ is another efficient representative of $[f]$, this is in particular the case when $[g] = [f^k], k \in \mathbb{Z}$.  By definition of the combinatorial equivalence in 2.1.3 this provides an infinite sequence of efficient representatives that are all combinatorially equivalent.
In the particular case of SC-representatives we have, by lemma 3.4 and proposition 3.5, a bijection between SC-representatives and admissible cut lengths. The $[f^k]$-action is described in term of cut lengths as follows:

\begin{lem}
Let $\mathcal{E} = (\Gamma, \Psi, h_{\Gamma} )$ be a SC-representative of $[f]$, relative to the leaf $u_{i_0}$ of admissible cut length $(l_1, ...,l_m)$.
Then for all $k \in \mathbb{Z}$, $(\lambda^k  l_1, ...,\lambda^k  l_m)$ is admissible, with corresponding SC-representative $(f^k)^*\mathcal{E}$, relative to $u_{\sigma^k (i_0)}$.
\end{lem}

\begin{proof} The fact that $(\lambda^k  l_1, ...,\lambda^k  l_m)$ is admissible
is immediate by iterating $[f]$ on the surface.  The $\mu^s$-length of a rectangle $f (R_i)$  is $\lambda$ times the $\mu^s$-length of $R_i$ for all rectangles. The boundary periodic points are the same and the length of the periodic Nielsen paths have all been multiplied by $\lambda$. \end{proof}

\subsection{Cycles in the set of SC-representatives.}  

In this section we prove the main technical result that SC-representatives of $[f]$ are organised in cycles whose structure will be studied in the next sections. We consider sequences of elementary SC-foldings. The first observation is that if $\mathcal{E'}$ is obtained from $\mathcal{E}$ by a sequence of elementary SC-foldings then the total length is strictly decreasing:
$l (\mathcal{E'}) < l (\mathcal{E})$ by proposition 4.2. The next lemma is necessary to prove that the total length converges to zero when an infinite sequence of elementary SC-foldings is applied.

\begin{lem}
Let $\mathcal{E}_1$ and $\mathcal{E}_2$ be two SC-representatives relative to the same separatrix, say $u_1$, and $q$ the order of the permutation $\sigma$. Assume that $l (\mathcal{E}_2 ) < l (\mathcal{E}_1 )$ and let
 $\alpha$ be the positive integer so that:
$l (\mathcal{E}_2 ) \in \big{[ } \frac{1}{\lambda^{\alpha  q}}  l (\mathcal{E}_1 ),  \frac{1}{\lambda^{(\alpha -1)  q}} .  
l (\mathcal{E}_1 ) \big{[} $.  Then the minimal edge length of $\mathcal{E}_2$ (of positive length) is at least 
$\frac{1}{\lambda^{(\alpha +1) . q}}$ times the minimal edge length of 
 $ \mathcal{E}_1$ (of positive length).
\end{lem}

\begin{proof}
We denote as above $\mathcal{L} ( \mathcal{E}_i)$, 
$\mathcal{W} ( \mathcal{E}_i)$ and $\mathcal{W}^{true} ( \mathcal{E}_i)$ for $i =1, 2$ the cut lines and the stable segments of the two SC-representatives.

{\em Case $\alpha = 1$}: First we claim that 
$\mathcal{W} ( \mathcal{E}_2)  \subset \mathcal{W}^{true} 
({f^{-2q}}^{* } \mathcal{E}_1)$.\\
The cut lengths of ${f^{-q}}^{* } \mathcal{E}_1$ are $\frac{1}{\lambda^q}$ times those of $ \mathcal{E}_1$ and, by assumption: $l(  \mathcal{E}_2) \geqslant \frac{1}{\lambda^q} l( \mathcal{E}_1)$. Therefore the cut lengths of ${f^{-q}}^{* } \mathcal{E}_1$ are smaller than those of $\mathcal{E}_2$. This implies that 
$\mathcal{W}^{true} ( \mathcal{E}_2)  \subset \mathcal{W}^{true} 
({f^{-q}}^{* } \mathcal{E}_1)$ and the $f^q$ images of the artificial singularities of $ \mathcal{E}_2$ belongs to 
$\mathcal{W}^{true} ( \mathcal{E}_2) $ by observation 2, thus 
$ f^q (\mathcal{W} ( \mathcal{E}_2) ) \subset \mathcal{W}^{true} 
({f^{-q}}^{* } \mathcal{E}_1)$ which completes the proof of the claim.

An edge $e$ of $\Gamma_2$ with positive length defines a rectangle in
$S - \left(\mathcal{L} ( \mathcal{E}_2)\bigcup \mathcal{W} ( \mathcal{E}_2)\right)$, its length is the $\mu^s$-length of a maximal unstable segment in the complement of $\mathcal{W} ( \mathcal{E}_2)$.
The above claim implies that this unstable segment is longer than a maximal unstable segment in the complement of 
$\mathcal{W}^{true} ({f^{-2q}}^{* } \mathcal{E}_1)$ and thus:
$l (e) \geqslant min_{e' \in \Gamma_1} l \big{(} {f^{-2q}}^{* }(e') \big{)}
= \frac{1}{\lambda^{2q}} min_{e' \in \Gamma_1} l (e ')$ . This completes the case $\alpha = 1$.

{\em Case $\alpha > 1$.} The preceding argument is applied to the representatives $\mathcal{E '}_1 = {f^{(\alpha - 1)q}}^{* } \mathcal{E}_1$
and $\mathcal{E }_2$.
\end{proof}

\begin{lem}
Let $\mathcal{E} = \mathcal{E}_0$ be a SC-representative of $[f]$ and 
$\mathcal{E}_1 ,   \mathcal{E}_2 , ...$ the sequence of SC-representatives obtained from $\mathcal{E}$ by successive elementary SC-foldings. 
The sequence  $l ( \mathcal{E}_i )$ is strictly decreasing and converges to zero.
\end{lem}
\begin{proof} The strict decreasing property is obvious from proposition 4.2.
Definition 4.4 and the observation that follows implies that the total length 
$l ( \mathcal{E}_i )$ drops by at least the minimal edge length of $\Gamma_i$.
Let $l_{min}^0 = min_{e \in \Gamma_0} l(e)$, lemma 4.7 implies that if
 $l ( \mathcal{E}_i ) \geqslant \frac{1}{\lambda^{\alpha q}} . 
l (\mathcal{E}_0 )$ then the SC-elementary folding applied to 
$\mathcal{E}_i $ makes the total length decreases by at least 
$ \frac{1}{\lambda^{(\alpha +1) q}} l_{min}^0$ and thus 
$l ( \mathcal{E}_i ) \rightarrow 0$.
\end{proof}

\begin{thm}[cycles exists]
\label{th:4.9}
Let $\mathcal{E} = \mathcal{E}_0$ be a SC-representative of $[f]$ and 
$\mathcal{E}_1 ,   \mathcal{E}_2 , ...$ the sequence of SC-representatives obtained from $\mathcal{E}$ by the successive sequence of elementary SC-foldings.  There exists a SC-representative $\mathcal{E}_p$ in this sequence so that: 
$l ( \mathcal{E}_p)  =  \frac{1}{\lambda^{ q}}  l (\mathcal{E}_0 )$ and thus
$\mathcal{E}_p  =  {f^{-q}}^{*} \mathcal{E}_0 $.
\end{thm}
\begin{proof}
Since the sequence $l ( \mathcal{E}_i ) $ is decreasing and converges to zero there is a $p > 1$ so that 
 $l ( \mathcal{E}_{p-1}) > \frac{1}{\lambda^{ q}}  l (\mathcal{E}_0 ) \geqslant l ( \mathcal{E}_{p})$.  We want to prove that 
$\frac{1}{\lambda^{ q}}  l (\mathcal{E}_0 ) =  l ( \mathcal{E}_{p})$. If this is the case then $ (l_1^{p}, ..., l_m^{p}) = \frac{1}{\lambda^{ q}}. (l_1^{0}, ..., l_m^{0})$ and, by lemma 4.6,  $\mathcal{E}_p  =  {f^{-q}}^{*} \mathcal{E}_0 $.

Let us assume by contradiction that the inequality is strict. We denote by
$L_{p-1} = l ( \mathcal{E}_{p-1})$ and
 $D = \frac{1}{\lambda^{ q}}  l (\mathcal{E}_0 )$. By assumption 
$l =  \frac{L_{p-1} - D}{1+ \frac{1}{\lambda} + \cdots + 
\frac{1}{\lambda^{ q-1}}} > 0 $.\\
We assume also that $\mathcal{E}_0$ is a SC-representative relative to the leaf  $u_1$, therefore all the SC-representatives in the sequence are relative to the same leaf, by definition of the elementary SC-folding.
We denote by $ (l_1^{p-1}, ..., l_m^{p-1})$ the admissible cut lengths of 
$ \mathcal{E}_{p-1}$. From observation 1 the cut length along the leaf $u_1$ is $ l_1^{0} =  \frac{ D}{1+ \frac{1}{\lambda} + \cdots + 
\frac{1}{\lambda^{ q-1}}}$ for $\mathcal{E}_0$ and 
$ l_1^{p-1} =  \frac{ L_{p-1}}{1+ \frac{1}{\lambda} + \cdots + 
\frac{1}{\lambda^{ q-1}}}$ for $\mathcal{E}_{p-1}$.
The definition of $l$ gives $ l_1^{0} = l_1^{p-1} - l$ and therefore 
$ l_{\sigma ^{- r}(1)}^{0} = l_{\sigma ^{- r}(1)}^{p-1} - \frac{1}{\lambda^{ r}}l$, for $r = 1, ..., q-1$.
These are admissible cut lengths since ${f^{-q}}^{*} \mathcal{E}_0 $ is a SC-representative. By assumption $l >0$ and $l < l'$, where $l'$ is the length of the elementary SC-folding at $\mathcal{E}_{p-1} $, this is a contradiction with  proposition 4.5.
\end{proof}

\section{Cycles: roots, symmetries and conjugacy.}

Specific properties of the set of SC-representatives are studied in this section.

\subsection{Cycles and the main theorem.}
\begin{thm}[Cycles structure]
\label{th:5.1}
Let $f$ be a pseudo-Anosov homeomorphism fixing $x_0$, let $(u_1, ..., u_m)$ be the set of unstable separatrices at $x_0$ that are permuted by $f$ and let $\sigma$ be the induced permutation of order $q$ of $\{1,\ldots m\}$.\\
(1) There exists a SC-representative $ \mathcal{E}_0 $ relative to $u_1$.\\
(2) The sequence of elementary SC-foldings $ \mathcal{E}_0, \mathcal{E}_1, \cdots $ defines a combinatorial cycle $ \mathcal{C}(u_1)$: 
there exists $p >1 $ such that $ \mathcal{E}_j $ is combinatorially equivalent to 
$ \mathcal{E}_i $ if $ i \equiv j $ modulo $[p]$.\\
(3) Every SC-representative relative to $u_1$ is combinatorially equivalent to a SC-representative in the cycle $ \mathcal{C}(u_1)$.\\
(4) If the separatrix $u_j$ belongs to the $\sigma$-orbit of $u_1$ then each SC-representative in the cycle $ \mathcal{C}(u_j)$ is combinatorially equivalent to a SC-representative in the cycle $ \mathcal{C}(u_1)$.

\end{thm}

\begin{proof} Items (1) and (2) are already proved. \\
(3) Let  $ \mathcal{E}_0, \mathcal{E}_1, \cdots,  \mathcal{E}_p = {f^{-q}}^{*} \mathcal{E}_0 $ denote the combinatorial cycle of theorem \ref{th:4.9}.
If $\mathcal{E}$ is a SC-representative relative to $u_1$ then there is an integer $\alpha$ so that 
$\lambda^{\alpha q} l (\mathcal{E}_0) \geqslant l (\mathcal{E}) >
\lambda^{(\alpha -1) q} l (\mathcal{E}_0)$.
The sequence 
${f^{\alpha q}}^{*}(\mathcal{E}_0), {f^{\alpha q}}^{*}(\mathcal{E}_1), \cdots, 
{f^{\alpha q}}^{*}(\mathcal{E}_p) = {f^{(\alpha -1) q}}^{*}(\mathcal{E}_0)$ represents the same combinatorial cycle.
The proof of theorem \ref{th:4.9} implies that $\mathcal{E}$ belongs to this sequence.\\
(4) Let $ \mathcal{E}_0$ be a SC-representative relative to $u_1$ and 
$ \mathcal{C}(u_1)$ the cycle obtained from $ \mathcal{E}_0$.
Let $ \mathcal{E}$ be a SC-representative relative to $u_j$ with 
$j = \sigma^k (1)$. Then by lemma 4.6 the SC-representative 
$ \mathcal{E '} = {f^{-k}}^{*} \mathcal{E}$ is a SC-representative relative to $u_1$ and the result comes from (3).
\end{proof}

Theorem \ref{th:5.1} describes a very simple structure on the set of SC-representatives of $[f]$ up to combinatorial equivalence: it is a union of cycles, one for each orbit of separatrices.

\subsection{Transition maps.}

The cycles in  theorem \ref{th:5.1} are defined up to combinatorial equivalence. In this paragraph we use the existence of a combinatorial equivalence to extract informations about the mapping class element.
Recall, from lemma 2.2, that the unstable separatrices with non zero cut length are in  one-to-one correspondence with the illegal turns of any SC-representative so the illegal turns are indexed via the name of the corresponding separatrix.

\begin{proposition} [transition map]
Let $\mathcal{E} = (\Gamma, \psi, h_{\Gamma}) $ ( resp. 
$\mathcal{E'} = (\Gamma', \psi ', h_{\Gamma'}) $ ) be a SC-representative relative to the separatrix $u_j$ (resp. $u_i$).
If $\mathcal{E}$ and $\mathcal{E'}$ are combinatorially equivalent under a map $ C: \Gamma \rightarrow \Gamma '$ then this map is unique.
Moreover the map $C$ induces a unique homeomorphism $g: S \rightarrow S$ 
that satisfies:\\
(i)  $g\circ h_{\Gamma} \simeq h_{\Gamma'} \circ C$ .\\
(ii) $g$ leaves invariant the pair of foliations
 $(\mathcal{F}^s , \mu^s) ; (\mathcal{F}^u , \mu^u)$, it rescales $\mu^u$ by a factor $\nu$ and $\mu^s$ by a factor $\nu^{-1}$ and it commutes with $f$.\\
(iii) $g$ maps the separatrix $u_j$ to $u_{i}$ and induces locally a rotation of angle $\frac{j - i}{m}2\pi$. There are two integers $b$ and 
$c \neq 0$ so that  $ g^{dc} = f^{bq}$, where $q$ is the local order of $f$ at $x_0$ and $d$ the local order of $g$ at $x_0$.\\
The homeomorphism $g$ is called the transition map.
\end{proposition}
 
Note that  (iii) implies that if the integer $b$ is zero then the homeomorphism $g$ is a finite order (a symmetry of $f$) and if $b \neq 0$ then it is a root of a power of $f$.\\
\begin{proof}
By definition, the combinatorial equivalence satisfies 
$\Psi ' \circ C = C \circ \Psi $ therefore it carries the unique illegal turn of order one of $\Psi$ to the unique illegal turn of order one of $\Psi '$. The image of the other edges is thus uniquely defined since $\Gamma$ and $\Gamma '$ are embedded graphs and $C$ respects the cyclic ordering at each vertex. The map $C$  induces a homeomorphism 
$\widetilde{C}: \mathcal{R}^{[\mathcal{E}]} \rightarrow 
 \mathcal{R}^{[\mathcal{E '}]} $, by mapping each rectangle $R(e), e \in \Gamma$ to the rectangle $R(C(e))$.
Each rectangle is equipped with a length and width structure, ie a pair of measured foliations, therefore $\widetilde{C}$ carries the pair of measured foliations of  $\mathcal{R}^{[\mathcal{E}]}$ to the pair of measured foliations of 
$\mathcal{R}^{[\mathcal{E'}]}$.
The length of the rectangle $R(e)$ is $l^{\mathcal{E}}(e)$ and the length of the rectangle $R(C(e))$ is  $l^{\mathcal{E '}}(C(e))$. The map $\widetilde{C}$ is chosen to be linear in each rectangle and it induces a rescaling of the unstable measured foliation by the factor
 $ \nu = \frac{l^{\mathcal{E '}}( C(e))}{l^{\mathcal{E }}(e)} = 
\frac{l (\mathcal{E '})}{l (\mathcal{E })} $.
We defined in paragraph 2.2.3 a quotient  map $\mathcal{IN}: 
\mathcal{R}^{[\mathcal{E}]} \rightarrow  S$ by isometric identification of the boundary $\partial \mathcal{R}^{[\mathcal{E}]}$ along the periodic Nielsen paths connecting two consecutive boundary periodic points of the map 
$ f_{\mathcal{E}}:  \mathcal{R}^{[\mathcal{E}]} \rightarrow 
 \mathcal{R}^{[\mathcal{E }]} $.
The homeomorphism $\widetilde{C}$ is, by construction, a conjugacy between 
$ f_{\mathcal{E}}$ and $ f_{\mathcal{E '}}$, therefore the boundary periodic points of $ f_{\mathcal{E}}$ are mapped to boundary periodic points of
$ f_{\mathcal{E '}}$  and periodic Nielsen paths are mapped to periodic Nielsen paths.
The homeomorphism $\widetilde{C}$ together with the maps $\mathcal{IN}: 
\mathcal{R}^{[\mathcal{E}]} \rightarrow  S$ and 
$\mathcal{IN '}: \mathcal{R}^{[\mathcal{E '}]} \rightarrow  S$ defines a homeomorphism $ g: S \rightarrow  S$, via the commutative diagram:
 
$$\begin{array}{ccc}
{\cal R}^{[\mathcal{E}]} & \stackrel{\widetilde{C}}{\longrightarrow} &{\cal R}^{[\mathcal{E}']} \\
{{}_{\cal{IN} } }{\downarrow} &                          & {\downarrow} {{}_{\mathcal{IN ' }} }\\
S & \stackrel{g}{\longrightarrow} & S
\end{array}
$$

that satisfies:\\
(1) $g \circ \mathcal{IN} = \mathcal{IN '} \circ \widetilde{C} $.\\
(2)  $g\circ h_{\Gamma} \simeq h_{\Gamma'} \circ C$.\\
(3) $g$ leaves invariant the two foliations $(\mathcal{F}^s , \mu^s)$ and  $(\mathcal{F}^u , \mu^u)$.\\
(4) $g$ is unique because of (3), see Fathi {\em et al.} \cite{flp79}.\\
Since both foliations are invariant under $f$ and $g$ then these two homeomorphisms commute, ie $g$ belongs to the centraliser of $f$.
In paragraph 4.1.2 we described the action of $[f^k]$ on the set of SC-representatives of $[f]$. The very same discussion applies to the action of the centraliser and we shall denote $ \mathcal{E '} = g^{*} \mathcal{E }$.

The surface $S$ together with the pair of measured foliations $(\mathcal{F}^s , \mu^s)$ and  $(\mathcal{F}^u , \mu^u)$ is equipped with an area form defined by the product of the measures $\mu^s$ and $\mu^u$. The homeomorphism $g$ preserves the total area and, since it rescales the measures $\mu^u$ by a factor $\nu$ so it rescales the measure $\mu^s$ by the factor $\nu^{-1}$. This completes the proof of (i) and (ii).
Our indexing convention of the illegal turns of $ \mathcal{E }$
and $ \mathcal{E '}$ by the index of the corresponding unstable separatrix of
$\mathcal{F}^u$ at $x_0$ implies that the cusp of 
$\mathcal{R}^{[\mathcal{E }]} $
labelled $j$ is mapped under $\widetilde{C}$ to the cusp labelled $i$ in 
$\mathcal{R}^{[\mathcal{E '}]} $.
In addition  $ \mathcal{IN}$ maps the periodic Nielsen path labelled $j$ to an initial segment of the separatrix $u_j$ and $ \mathcal{IN '}$ maps the periodic Nielsen path labelled $i$ to an initial segment of the separatrix $u_{i}$. So condition  (1) above implies that $g$ maps $u_j$ to $u_{i}$.
In addition $g$ fixes the marked point $x_0$ and is an orientation preserving homeomorphism on $S$ so it acts locally like a rotation around $x_0$ sending $u_j$ to $u_{i}$.
This is a $\frac{j - i}{m}. 2\pi$ rotation of order $d$ in $\mathbb{Z} / m\mathbb{Z}$. Finally we observed that $g$ belongs to the centraliser of $f$, it has order $d \in \mathbb{Z} / m\mathbb{Z}$ locally around $x_0$ whereas $f$ has local order $q$. Therefore it satisfies $g^{dc} = f^{qb}$, where $b$ and $c$ are integers and $c\neq 0$.
\end{proof}

\subsection{Roots.}

From theorem \ref{th:5.1} and proposition 5.2  the set $\mathcal{ SC} [f]$ of SC-representatives of $[f]$ is a union of cycles
$\mathcal{C} (\mathcal{O}_i)$, one for each $f$-orbit of unstable separatrices at the marked fixed point $x_0$. Each cycle is obtained from any of its representative by a finite sequence of SC-elementary folding operation.
Two types of situations could simplify this set of cycles:\\
- A cycle could be shorter than expected (see propositions \ref{prop:5.3} and \ref{prop:5.5}).\\
- Two a priori different cycles could be equivalent (see propositions \ref{prop:5.4} and \ref{prop:5.6}).\\
These simplifications reflect the existence of a combinatorial equivalence that is different from the one we use to prove the existence of cycles, namely the fact that $ \mathcal{E }$ is equivalent to  ${ f^{-q}}^* \mathcal{E }$. Proposition 5.2 then implies the existence of a homeomorphism in the centraliser of $[f]$ that is not at power of $f$.


The permutation $\sigma$ induced by $f$ on the set 
$\{u_1,\ldots,u_m\}$ of unstable separatrices at $x_0$ acts like a 
rotation of angle $2\pi \frac{r}{m}$ for some $r\in\{0\ldots m-1\}$, 
called the {\em rotation number} of $f$ at $x_0$.

 If $ \mathcal{E }$ is a SC-representative of $[f]$ 
relative to $u_i$, the combinatorial cycle ${\cal C}(u_i)$ of SC-representatives relative to $u_i$ of $[f]$ is defined by 
$ \mathcal{E }= \mathcal{E }_0,  \mathcal{E }_1, \ldots, 
\mathcal{E }_{p_i}={f^{-q}}^* \mathcal{E }_0$.
 If there exists an integer $p$ smaller than $p_i$ such that $ \mathcal{E }_p$ is 
combinatorially equivalent to $ \mathcal{E }_0$, let $p'_i$ be the smallest such 
integer. One says that ${\cal C}(u_i)$ admits a {\em combinatorial 
primitive subcycle} ${\cal C}'(u_i):  \mathcal{E }= \mathcal{E }_0,  
\mathcal{E }_1, \ldots, \mathcal{E }_{p'_i}$. In 
this case, $p'_i$ divides $p_i$ and  the cycle ${\cal C}(u_i)$ 
is obtained by looping several times around the primitive subcycle 
${\cal C}'(u_i)$.

Observe that if $g$ is a $k^{th}$-root of $f$ fixing 
$x_0$, then it has the same invariant foliations as $f$, and the permutation 
$\sigma'$ induced by $g$ on the separatrices $(u_1,\ldots,u_m)$ satisfies ${\sigma '}^k=\sigma$.

The following four results give necessary and sufficient conditions 
for the map $f$ to admit roots fixing $x_0$. We  distinguish 
different cases, according to the rotation number of $f$ and its 
roots at the point $x_0$.

\begin{proposition}\label{prop:5.3}
The notations are like in theorem \ref{th:5.1}. We assume $q=1$ 
($f$ has rotation number zero at $x_0$). The following 
conditions are equivalent:\\
i) There exists a pseudo-Anosov homeomorphism $g$ fixing $x_0$ 
and an integer $k>1$ such that $f=g^k$ and the rotation number of $g$ 
at $x_0$ is zero.\\
ii) For every separatrix $u_j$ at $x_0$, the cycle 
${\cal C}(u_j)$ admits a primitive subcycle ${\cal C}'(u_j)$.\\
iii)   For some separatrix $u_j$, the cycle ${\cal C}(u_j)$ admits a primitive
subcycle ${\cal C}'(u_j)$.\\
When iii) is satisfied and 
${\cal C}'(u_j) = \mathcal{E }_0,\ldots,\mathcal{E }_{p'_j}$, then a root of $f$ is the inverse of the transition map from $\mathcal{E }$ to 
$\mathcal{E }_{p'_j}$.
\end{proposition}

\begin{proof}
$i)\Rightarrow ii)$ If $f=g^k$ then $g$ is pseudo-Anosov and its 
growth rate is $\nu=\lambda^{1/k}$. Let $\mathcal{E }$ be a SC-representative 
relative to some leaf $u_j$. The cycle ${\cal C}(u_j)$ is given 
by $\mathcal{E } = \mathcal{E }_0, \mathcal{E }_1,\ldots,
\mathcal{E }_{p_j}={f^{-1}}^*\mathcal{E }$. Then $\mathcal{E }'= {g^{-1}}^*\mathcal{E }$ is 
a SC-representative relative to $u_j$, since $g$ has rotation number $0$ at $x_0$ and its total length is 
$l(\mathcal{E }')=\frac{1}{\nu}l(\mathcal{E })$. It follows from theorem \ref{th:5.1} 
$(3)$ that $\mathcal{E }'$ is equivalent to $\mathcal{E }_p$ for some $p$. Since  $1<\nu<\lambda$ then $0<p<p_j$ (see the proof of theorem \ref{th:4.9}) and ${\cal C}(u_j)$ admits a subcycle. Since $j$ was arbitrary, ii) is proved.

$ii)\Rightarrow iii)$ is obvious.

$iii)\Rightarrow i)$ Assume that condition $iii)$ is satisfied 
and denote by $g$ the transition map from $\mathcal{E }$ to 
$\mathcal{E }_{p'_j}$. As observed above, the cycle ${\cal C}(u_j)$ is a multiple of the subcycle ${\cal C}'(u_j)$ several times. Hence there is an integer $c>1$ such that $g^c=f^{-1}$. This proves $i)$. 
 \end{proof}

\begin{proposition} \label{prop:5.4}
The notations are like in theorem \ref{th:5.1}.
Assume $q=1$ ($f$ has rotation 
number zero at $x_0$). The following conditions are equivalent:\\
i) There exists a pseudo-Anosov homeomorphism $g$ fixing $x_0$ 
such that $f=g^k$ and the rotation number of $g$ at $x_0$ is 
non-zero. Moreover $f$ is not the power of a map with zero rotation 
number.\\
ii) For every $i$, the cycle ${\cal C}(u_i)$ relative to $u_i$ 
has no combinatorial subcycles. There exists $j\neq i$ such that the 
cycle ${\cal C}(u_j)$ is combinatorially equivalent to the cycle 
${\cal C}(u_i)$ and for two equivalent representatives 
${\mathcal{E }} \in {\cal C}(u_i)$ and ${\mathcal{E '}} \in {\cal C}(u_j)$, the lengths satisfy: 
$\frac{l (\mathcal{E '})}{l (\mathcal{E })} \neq 1$ . 
The transition map $g$ from $\mathcal{E }\in{\cal C}(u_1)$ 
to some $\mathcal{E }'\in{ \cal C}(u_j)$ is a root of $f$.\\
\end{proposition} 

\begin{proof}
$ii)\Rightarrow i)$ Since ${\cal C}(u_i)$ is equivalent to ${\cal C}(u_j)$ and has no subcycle there is only one $\mathcal{E '}\in {\cal C}(u_j)$ that is equivalent to $\mathcal{E }\in {\cal C}(u_i)$.  The cycle ${\cal C}(u_i)$ might be equivalent to several ${\cal C}(u_j)$. The transition map $h$ from $\mathcal{E }$  to $\mathcal{E '}$ maps $u_i$ to $u_j$ and 
$i-j \in \mathbb{Z} / m\mathbb{Z}$ is of order $k_j  > 1$. By proposition 5.2 (iii)  $ f ^b = h^{ck_j}$, and thus there are $v$ and  $k$ relatively prime so that $ f ^v = h^{k}$. Let $x$ and $y$ be such that $ vy + kx =1$ and let 
$g = f^x h^y$ then $g^k = f^{xk} h^{yk}  = f^{xk} f^{yv} = f$, this map $g$ is a transition map as a composition of such.

$i)\Rightarrow ii)$ Let $g$ be a root of $f$ fixing $x_0$ with 
non-zero rotation number at $x_0$, so that $f=g^k$. The rotation 
number of $g$ at $x_0$ has exactly order $k$, for if it was smaller 
than $k$ then $f$ would have a root with zero rotation number. Let 
$\mathcal{E }$ be a SC-representative relative to $u_i$ and $u_j=g(u_i)$, then 
$\mathcal{E }' = g^*\mathcal{E }$ is a SC-representative for $f$ relative to $u_j$ and the transition map from $\mathcal{E }$ to $\mathcal{E }'$ is $g$. The length condition is given by proposition 5.2 (ii).
\end{proof}

We now treat the cases when $f$ has non-zero rotation number at 
$x_0$.

\begin{proposition} \label{prop:5.5}
The notations are like in theorem \ref{th:5.1}.
Assume that $f$ has 
non-zero rotation number at $x_0$. The following conditions are 
equivalent:\\
i) There exists a pseudo-Anosov map $g$ fixing $x_0$ such that 
$f=g^k$ and the rotation number of $g$ at $x_0$ is the same as $f$.\\
ii) For every $j$, the cycle ${\cal C}(u_j)$ relative to $u_j$ 
has a subcycle. Moreover if $\mathcal{E }$ is any SC-representative relative to 
$u_j$, let $h$ be the transition map from $\mathcal{E }$ to the first 
representative equivalent to $\mathcal{E }$ in the primitive subcycle ${\cal C}'(u_j)$ 
starting from $\mathcal{E }$. Let $c>1$ be the integer such that $h^c=f^{-q}$. 
There exists $1\leq b < c$ such that the map $h^b$ admits a $q-$th root 
fixing $x_0$ with non-zero rotation number that is also a root of 
$f^{-1}$.
\end{proposition}

\begin{proof}
$i)\Rightarrow ii)$ The map $g^q$ is a root of $f^q$ and has  
rotation number zero. Hence if $\mathcal{E }$ is a SC-representative relative to 
$u_j$, then ${g^{-q}}^*\mathcal{E }$ is a SC-representative relative to $u_j$ 
and its total length satisfies $l(\mathcal{E }) > l({g^{-q}}^*\mathcal{E }) > l({f^{-q}}^*\mathcal{E })$. 
The cycle relative to $u_j$ has a subcycle, and the map $g^{-q}$ is a 
power of the transition map $h$ from $\mathcal{E }$ to the first 
representative equivalent to $\mathcal{E }$. It is of the form $h^b$ with 
$1\leq b < c$ (in fact $b$ divides $c$).

$ii)\Rightarrow i)$ The $q$-th root of a map with zero rotation 
number has a rotation number whose order divides $q$. But a root of 
$f^{-1}$ has rotation number whose order is a multiple of $q$. 
\end{proof}

\begin{proposition}\label{prop:5.6}
The notations are like in theorem \ref{th:5.1}.
Assume  that $f$ has 
non-zero rotation number with order $q$ at $x_0$. The following conditions are 
equivalent:\\
i) There exists a pseudo-Anosov map $g$ fixing $x_0$ such that 
$f=g^k$ and the rotation number of $g$ at $x_0$ has order $aq$ with 
$a >1$.\\
ii) Let $j = \frac{m}{aq}$, the cycle ${\cal C}(u_{j+1})$ is 
combinatorially equivalent to ${\cal C}(u_1)$. Let $\mathcal{E }$ be a SC-
representative relative to $u_1$. There exists an integer $k$ 
relatively prime with $aq$ such that one transition map from $\mathcal{E }$ to a representative in the cycle relative to $u_{kj+1}$ is a root of $f$.

\end{proposition}

For the practical determination of a root, it is useful to note the 
following. Assume that the  property $ii)$ above
is satisfied. Let $g$ be the transition map from $\mathcal{E }$ to some SC-representative relative to $u_{j+1}$, let $h$ be the transition map 
from $\mathcal{E }$ to the first SC-representative equivalent to $\mathcal{E }$ in the cycle $\mathcal{E } = \mathcal{E }_0, \mathcal{E }_1,\ldots,
\mathcal{E }_p={f^{-q}}^*\mathcal{E }$. There exist integers $b$ 
and $c$ such that $h^b=f^q$ and $g^a=f h^c$. Any transition map from 
$\mathcal{E }$ to another SC-representative relative to some leaf $u_{kj+1}$ is of the form $h^ug^v$. The roots of $f$ are to be searched among such 
maps, with the restriction  $1\leq v\leq aq-1$. Furthermore the 
value of $u$ can only take finitely many values since the growth 
rate of a root $g$ is smaller than the growth rate of $f$.

\begin{proof} 
$i)\Rightarrow ii)$ The element $j\in  \mathbb{Z}/m \mathbb{Z}$ is a multiple of the rotation number of $g$, hence there exists $b>0$ such that $g^b$ has rotation number $j$. Then ${g^b}^*\mathcal{E }$ is a SC-
representative relative to $u_{j+1}$ and is combinatorially 
equivalent to $\mathcal{E }$. This proves the first part of $ii)$. The rotation 
number of $g$ at $x_0$ is of the form $kj$, $k$ prime with $aq$. 
Hence $g^*\mathcal{E }$ is a SC-representative relative to $u_{kj+1}$.

$ii)\Rightarrow i)$ is similar to the previous proofs.
\end{proof}

\subsection{Symmetries.}

When a finite order orientation preserving homeomorphism 
$g$ fixing $x_0$ satisfies $g\circ f=f\circ g$, we say that the 
pair of foliations $(\mathcal{F}^s,\mathcal{F}^u)$ and the homeomorphism $f$ {\em admit the symmetry} $g$. In this case, the foliations are invariant under $g$ and the dilatation factor of $g$ is $1$.

In the previous paragraph, we studied all the possible cases when a 
transition map $g$ between two SC-representatives for $f$ is a root 
of $f$. This happens when the dilatation factor of $g$ is 
$\nu=\lambda^{\frac{1}{k}}$, and the rotation number of $g$ at $x_0$ 
is a "$k$-th root" of the rotation number of $f$. Here, we show that when 
$\nu=1$ and the rotation number of $g$ at $x_0$ is non-zero, then $f$ 
admits the symmetry $g$.

\begin{proposition}
The following conditions are equivalent: \\
i) There exists a non trivial finite order homeomorphism $g$ 
fixing $x_0$ such that $f\circ g = g\circ f$.\\
ii) For every leaf $u_j$ and every SC-representative of $f$ 
relative to $u_j$, there exists $j'\neq j$ and a SC-representative 
relative to $u_{j'}$ that is combinatorially equivalent to $\mathcal{E }$ and 
has the same total length.\\  
iii) For some $j\in\{1\ldots m\}$ and some SC-representative 
of $f$ relative to $u_j$, there exists $j'\neq j$ and a SC- 
representative relative to $u_{j'}$ that is combinatorially 
equivalent to $\mathcal{E }$ and has the same total length.\\  
If condition $iii)$ is satisfied, then a finite order 
homeomorphism commuting with $f$ is given by the transition map from 
$\mathcal{E }$ to $\mathcal{E }'$.
\end{proposition}

\begin{proof}
$i)\Rightarrow ii)$ if $g$ is a finite order homeomorphism that commutes with 
$f$ then $\mathcal{E }' = g^*\mathcal{E }$ is a SC-representative for $f$. It has the same total length: $l(\mathcal{E }') = l(\mathcal{E })$. The homeomorphism $g$ permutes the 
separatrices at $x_0$ in a non trivial fashion, since otherwise $g$ would be the identity. Hence $\mathcal{E }'$ is 
a SC-representative of $f$ relative to the leaf $u_{j'}=g(u_j)$ and 
$j'\neq j$.

$ii)\Rightarrow iii)$ is obvious.

$iii)\Rightarrow i)$ Let $g$ be the transition map from $\mathcal{E }$ to 
$\mathcal{E }'$. Let $k$ be such that $k(j'-j)\equiv 0[m]$. Then ${g^k}^*\mathcal{E }$ is a 
SC-representative relative to $u_j$ and it has the same total length 
as $\mathcal{E }$. Hence ${g^k}^*\mathcal{E } = \mathcal{E }$ by observation 1 of paragraph 4.1, that is $g^k=id$.
\end{proof}

\subsection{The conjugacy problem.}
Theorem \ref{th:5.1} yields another solution to the conjugacy problem among pseudo-Anosov elements of the mapping class group, in the particular case when one marked point is fixed. 

Let $\mathcal{SC}(f)$ denote the set of {\em combinatorial SC-representatives} of $[f]$, ie the set of SC-representatives up to combinatorial equivalence. 

\begin{proposition}
The set $\mathcal{SC}(f)$ is a complete conjugacy invariant. This set is a union of cycles that is computable.
\end{proposition}

\begin{proof} If $[f]$ and $[g]$ are conjugate then the sets $\mathcal{SC}(f)$  and $\mathcal{SC}(g)$  are the same. Indeed if $g = h\circ f \circ h^{-1}$ and 
$\mathcal{E}$ is a SC-representative of $[f]$ then 
$h^*\mathcal{E}$ is a SC-representative of $[g]$ and $h^*\mathcal{E}$ is combinatorially equivalent to $\mathcal{E}$, thus 
$\mathcal{SC}(f) \subset \mathcal{SC}(g)$ and by symmetry 
$\mathcal{SC}(g) \subset \mathcal{SC}(f)$.

Conversely if $\mathcal{SC}(f) = \mathcal{SC}(g)$ then $[f]$ and $[g]$ are conjugate. 
The computability of the cycle structure is given in the next section.
\end{proof}

{\em Remark.} In Los \cite{lo96} the principle for solving the conjugacy problem (for irreducible free group outer automorphisms) was similar but using the whole set of combinatorial efficient representatives as a complete invariant rather than the small subset of SC-representatives. The computability was stated as a "connectedness" property saying that any two efficient representatives are connected by a sequence of folding, collapsing or the converse of a folding.  Here the set of complete invariant is much smaller and describing each cycle is quite simple since only foldings are necessary and in a specific ordering. Furthermore there is a strong connectedness result, that will appear in the next section. It  states the existence of one combinatorial efficient representative, say $(\Gamma_0, \Psi_0)$, that is not in general a SC-representative so that, for each element $\mathcal{E} \in \mathcal{SC}(f)$, $(\Gamma_0, \Psi_0)$ is connected to $\mathcal{E}$ by a well defined sequence of folding operations.

\section{Algorithmic aspects.}

From theorem \ref{th:5.1} and propositions 5.3--5.7 we know how to solve the roots and the symmetry problems. The goal is to obtain an effective algorithmic solution, it remains to show that all the quantities described in the previous paragraphs are computable.
First we know from Bestvina and Handel \cite{BH95} that an efficient representative is obtained after a finite algorithm, called the {\em train track algorithm}, starting from any topological representative of the mapping class $[f]$. 
The train track algorithm provides also a way to check whether $[f]$ is a pseudo-Anosov class and admits a marked fixed point.

The initial data here is an efficient representative $\mathcal{E }_0 = (\Gamma, \Psi, h_{\Gamma} )$  of $[f]$.

In order to use theorem \ref{th:5.1} and the propositions 5.3--5.7, several new transformations are necessary to describe completely the cycles of SC-representatives.
To this end we describe how to perform the following steps:\\
- Transform any efficient representative into a SC-representative, relative to a separatrix $u_j$.\\
- Describe one cycle $\mathcal{C} (u_j)$.\\
- Describe all the cycles $\mathcal{C}$.\\
- Compare the cycles.\\
The first step is a preliminary algorithm that requires a new operation: {\em glueing a $\sigma$-orbit}. The second step only requires the SC-folding operations, but we have to define this operation combinatorially rather than metrically as in section 4. The third step needs a transformation from one cycle  $\mathcal{C} (u_j)$ to another cycle $\mathcal{C} (u_k)$ and this requires also a new operation, called {\em splitting an infinitesimal turn}.
The final step requires checking the combinatorial equivalence on the one hand and the length on the other hand. The combinatorial equivalence is easy to check (in principle). For the length we define a "based efficient representative" that is not a SC-representative but has the property that all cycles are obtained from this single "based point" only by glueing and folding operations.

\subsection{Glueing a $\sigma$-orbit of illegal turns.}
The glueing operation introduced in this paragraph is new. The principle is quite simple: starting from an efficient representative 
$\mathcal{E } = (\Gamma, \Psi, h_{\Gamma} )$  of $[f]$ we want to glue an orbit of periodic Nielsen paths in order to find a SC-representative.
A priori $\mathcal{E }$ has several orbits of Nielsen paths associated to orbits of illegal turns, some of these turns are tangencies of order one. The goal is to transform all of these orbits but one into an orbit of infinitesimal turns.

We assume first that a length function is given on $\Gamma$ as in section 2.2. The cut lengths $( l_1, ..., l_m)$ associated to $\mathcal{E }$ are given by the map $\mathcal{L}_{\mathcal{E }}$ of lemma 2.2 from the orbits of periodic Nielsen paths. We denote by $\sigma$ the permutation of the indexes of the periodic Nielsen paths (under $\Psi$ ) as well as the permutation (under $f$) of the corresponding indexes of the separatrices at the fixed marked point $x_0$. We denote by $\mathcal{O}$ one $\sigma$-orbit for which $l_i > 0$ and we assume that $\sigma$ is not transitive.

{\verb"Step 1."} The first combinatorial step is to find  the  periodic orbits on the boundary, as well as all the associated periodic Nielsen paths $N_{t_i}$. Those are in a one-to-one correspondence with the illegal turns $t_i$, as described in paragraph 2.2.3. 
The Nielsen path $N_{t_i}$ is considered either as a path between two consecutive boundary periodic points along the boundary 
$\partial \mathcal{R}^{\mathcal{E }}$ and passing through the cusp $C_{t_i}$ associated to the illegal turn $t_i$ or as the corresponding edge path in $\Gamma$. Recall also that the Nielsen path is the concatenation $N_{t_i} = A_{t_i}^{-1} B_{t_i}$, where $A_{t_i}$ and $ B_{t_i}$ are the paths starting at the cusp $C_{t_i}$ toward the boundary periodic points along $\partial \mathcal{R}^{\mathcal{E }}$. We assume that $\mathcal{E }$ has more than one orbit of Nielsen paths.

We subdivide the two paths $A_{t_i}$ and $ B_{t_i}$ at a finite collection of points $x \in D(t_i) \subset \Gamma$. The two paths $A_{t_i}$ and $ B_{t_i}$ have the same length, by definition. A point $x\in A_{t_i}\bigcup B_{t_i}$ belongs to $D(t_i)$ either if $x$ is a vertex or if the distance from 
$x$ to $t_i$ along $A_{t_i}$ (resp. $ B_{t_i}$)  is the same than the distance from a vertex $y$ to $t_i$ along $B_{t_i}$ (resp. $ A_{t_i}$). The boundary periodic points at the end of $A_{t_i}$  and $ B_{t_i}$ also belong to $D(t_i)$. This metric description of the set $D(t_i)$ is replaced in practise by an easy combinatorial description using the map $\Psi$ and it's iterates.

{\verb"Step 2."} We subdivide $\Gamma$ at each point $x \in D(t_i)$
for all $t_i$ in one $\sigma$-orbit $\mathcal{O}$. The resulting graph is denoted $\Gamma_s$.

{\verb"Step 3."} The map $\Psi: \Gamma \rightarrow \Gamma$ induces a well defined map  $\Psi_s: \Gamma_s \rightarrow \Gamma_s$.
By definition of the subdivision we obtain:

\begin{lem}
The map $\Psi_s: \Gamma_s \rightarrow \Gamma_s$ satisfies the following property: for every illegal turn $t_i$ in the $\sigma$-orbit $\mathcal{O}$ and each edge $e_{\alpha}$ on the path $ A_{t_i}$, there is an edge $e_{\beta}$ on the path 
$ B_{t_i}$, at the same distance from $t_i$ so that $\Psi_s (e_{\alpha}) $ and 
$\Psi_s (e_{\beta}) $ are either equal or are subarcs $e'_{\alpha}$, $e'_{\beta}$ along $ A_{t_{\sigma (i)}}$ and $ B_{t_{\sigma (i)}}$ at the same distance from the illegal turn $t_{\sigma (i)}$. $\square$
\end{lem}

{\verb"Step 4."} Let $\mathcal{E }_s = (\Gamma_s, \Psi_s, h_{\Gamma_s} )$
be the above subdivided efficient representative of $[f]$. 
We define a new graph $\Gamma '$ by identifying all the pairs 
$(e_{\alpha}, e_{\beta})$ of lemma 6.1. The embedding 
$ h_{\Gamma_s} $ induces a well defined embedding $ h_{\Gamma '} $.
Let  $g: \Gamma_s \rightarrow \Gamma '$ denote the quotient map.
We observe,  by lemma 6.1, that if $g(e) = g(e')$ for two edges $(e, e')$ of $\Gamma_s$ then either $\Psi_s (e) = \Psi_s (e')$ or 
$ g( \Psi_s (e)  ) = g(  \Psi_s (e') )$. Therefore there is a well defined map 
$\Psi ': \Gamma ' \rightarrow \Gamma '$ so that:
$ \Psi ' \circ g = g \circ \Psi _s$.

\begin{proposition}
The triple $\mathcal{E }' = (\Gamma ', \Psi ', h_{\Gamma '} )$ is an efficient representative of $[f]$, obtained from $\mathcal{E }$ by {\em glueing } the $\sigma$-orbit $\mathcal{O}$. The cut length 
$l_{\mathcal{E }'} = (l'_1, ..., l'_m)$ satisfies $l'_i = 0$ if $i \in \mathcal{O}$
and $l'_i = l_i$ if $i \notin \mathcal{O}$.
\end{proposition}

The fact that $\mathcal{E }'$  is efficient is clear from the construction. Indeed no new illegal turn has been created.  The properties of the cut length is also clear: the periodic Nielsen paths corresponding to the orbit $\mathcal{O}$ has been replaced by a periodic orbit of infinitesimal turns, therefore of length zero.
The length of the other periodic Nielsen paths has not been affected, although the paths themselves might be very different. $\square$

In practise, this operation can be performed as a sequence of (classical) folding operations of illegal turns $t_i$ that belong to the orbit $\mathcal{O}$.

\subsection{Elementary SC-folding.}

Let us describe the SC-folding operation in a combinatorial setting. We assume that a SC-representative with respect to a separatrix, say $u_1$, is given.
Let $t_1 = (a,b)$ be the illegal turn of order one that is associated with $u_1$. We consider as above the Nielsen path
$N_{t_1} = A_{t_1}^{-1} B_{t_1}$.
Along one of the paths $A_{t_1}$ or $B_{t_1}$, there is a first vertex $v_1$ with at least 3 gates by lemma 4.3 and the construction of paragraph 2.2. Let  $l_0$ denote the length of the arc $[t_1, v_1]$ along $A_{t_1}$ or $B_{t_1}$.
We apply a first sequence of standard folding operations, as described by Bestvina and Handel \cite{BH95}, until the illegal turn of order one associated to $u_1$ is based at the corresponding vertex with at least 3 gates. The number of such operations is uniformly bounded, in term of the combinatorial length of $A_{t_1}$ and $B_{t_1}$. We obtain an efficient representative $\mathcal{E}'$ which, in general, is not a SC-representative.
Indeed, if $\sigma \neq id$ then the turn $t_{\sigma ^{-1} (1)} \neq t_1$ associated to 
$f^{-1} (u_1) \neq u_1$ is now of order one.
We apply next a sequence of standard foldings at 
$t_{\sigma ^{-1} (1)}$ until it becomes of order greater than one and we iterate the sequence for 
$t_{\sigma ^{-p} (1)}, p = 2, ..., q-1$, until all of these turns are of order greater than one.
Each of these sequences at $t_{\sigma ^{-p} (1)}$ has a uniformly bounded number of elementary folding operations.
Observe that the above sequence of foldings is uniquely defined by the initial SC-representative. 
The resulting efficient representative satisfies:

\begin{lem}
The above sequence of standard folding operations:
 $\mathcal{E} \rightarrow \mathcal{E}' .... \rightarrow \mathcal{E}^{(n')} = \mathcal{E}^{1}$ results in a SC-representative relative to $u_1$.  The number of standard folding operations in this sequence is uniformly bounded. The global operation 
 $\mathcal{E} \rightarrow \mathcal{E}^{1}$ is the elementary SC-folding at $\mathcal{E}$ of length $l_0$.$\square$
 \end{lem}

\subsection{Splitting an orbit of infinitesimal turns.}

In this paragraph we define a transformation that can be applied when an efficient representative $\mathcal{E}$ has an orbit $\mathcal{O}$ of separatrices whose cut length is zero. It only makes sense if the permutation $\sigma$ is not transitive. The goal is to obtain another efficient representative for which the orbit $\mathcal{O}$ has non zero cut length.

From lemma 2.2 the zero cut length property arises when a vertex $v_1$ with more than 3 gates is a boundary periodic point. In this case the separatrices of the orbit $\mathcal{O}$ are in bijection with some infinitesimal turns obtained by the train track construction of paragraph 2.2.2 at the orbit of $v_1$. 
In what follows we will either discuss the efficient representative $\mathcal{E}$ or the corresponding train track map $(\widetilde{\Psi}, \tau) (\mathcal{E})$. The zero cut length situation corresponds for $\mathcal{E}$ to a periodic orbit of vertices $ v_1, ..., v_{q}$ with more than 3 gates.
For the train track map $(\widetilde{\Psi}, \tau) (\mathcal{E})$ there is also a periodic orbit of vertices $ w_1, ..., w_{q}$ with exactly 3 gates. At each $w_i$ there is an infinitesimal turn $t_i^{\epsilon} = (\epsilon_i^{L} , \epsilon_i^{R})$
(L /R stands for left and right) on one side and one gate 
$g_i = \{a_i^1, ..., a_i^k\}$ on the other side.

{\verb"Step 1: preliminary foldings."} 
If each gate $g_i$ is reduced to a single edge then $w_i$ has valency 3 and the next transformations are easier to apply. 
If some of the gates has more than one edge then we apply some foldings leading to the previous simple situation with valency 3 at each vertex.
The periodicity assumption on the infinitesimal turns $t_i^{\epsilon}$ implies that one edge of the gates $g_i$ has periodic initial segment under 
$\widetilde{\Psi}$ (the same is also true for $\Psi$), ie:\\
 $\widetilde{\Psi} (a_1^{j_1}) = a_2^{j_2}M_1,  \widetilde{\Psi} (a_2^{j_2}) = a_3^{j_3}M_2 , ..., \widetilde{\Psi} (a_q^{j_q}) = a_1^{j_1}M_q $,\\ 
where the $a_i^{j_i}$ are oriented from the $w_i$ see figure 7.
Some of the edge paths $M_i$ might be trivial but not all of them since 
$\widetilde{\Psi}$ is a strict dilatation on each non infinitesimal edge.\\
We choose one gate (with more than one edge), say $g_1$, and we consider the first index $i$, in the above sequence ordered by the orbit, for which $M_i$ is non trivial. Let $v$ be the initial vertex of $M_i$,  we subdivide $a_i^{j_i}, a_{i-1}^{j_{i-1}} , . ...a_1^{j_1}$ at the preimages of $v$. This defines vertices $ v'_i, ..., v'_1$ of valency 2 on 
$a_i^{j_i},. ...a_1^{j_1}$. Then we apply a standard folding of all the edges in the gate $g_1$, up to $v'_1$ and then at all the preimage gates. Note that several other choices of subdivisions and preliminary folding operations are possible here. Our choice is made to obtain a periodic orbit of valency 3 vertices.

After this sequence of preliminary foldings we obtain a new train track map 
$(\widetilde{\Psi '}, \tau ')$ and a new efficient representative $\mathcal{E '}$
so that the gates $g'_i$ corresponding to $g_i$ are all reduced to a single edge $a'_i$.
The infinitesimal turns  $t_i^{\epsilon} = (\epsilon _i^{L} , \epsilon _i^{R})$
has not been affected by these foldings and are still periodic, ie: 
$(\epsilon _1^{L} , \epsilon _1^{R}) \rightarrow (\epsilon _2^{L} , \epsilon _2^{R}) ... \rightarrow (\epsilon _q^{L} , \epsilon _q^{R})$, under $\widetilde{\Psi '}$.\\
Similarly the edges $a'_i$ have periodic initial segments:
$\widetilde{\Psi '} (a'_1) = a'_2 M'_1,  ..., \widetilde{\Psi '} (a '_q) = a '_1M'_q $.

{\verb"Step 2: splitting."} 
The topological idea is to ``cut along $a'_1$ ''. In practise, in the simple situation obtained after the preliminary foldings, ie when all the vertices 
$w_1, ..., w_q$ have valency 3, this operation is nothing else than a ``collapsing".
Let $\tau ''$ be the graph obtained from $\tau'$ by collapsing the edge $a'_1$ and we assume that the edges of $\tau''$ keeps the same name and orientation
than in  $\tau'$.
This graph is embedded in the surface and the embedding $h_{\tau ''}$ is induced by the embedding $h_{\tau '}$.
We define a new map $\widetilde{\Psi ''}$ on $\tau ''$ as follows:

 $\widetilde{\Psi ''} (\epsilon _1^{L}) = ( \epsilon _2^{L} a'_2 M'_1)^*$,
$\widetilde{\Psi ''} (\epsilon _1^{R}) = ( \epsilon _2^{R} a'_2 M'_1)^*$ and
$\widetilde{\Psi ''} (e) = (\widetilde{\Psi '} (e))^*$, for all other edges,

where  $(W)^*$ is the word, representing an edge path in $\tau''$, obtained from the word $W$ representing an edge path in $\tau'$ by removing all occurrences of the letter $a'_1$. The labelling and the orientation of the edges are given in figure 7.
Let us make a couple of observations:\\
- The edges  $\epsilon _i^{L/R}, i = 1, ...,q$ are not infinitesimal anymore.\\
- All turns $(\epsilon _i^{L}, \epsilon _i^{R})$ are now illegal and 
$(\epsilon _1^{L}, \epsilon _1^{R})$ is of order one.\\
- The new map $ \widetilde{\Psi ''}: \tau'' \rightarrow \tau''$ is a train track map.\\
Indeed no edge $e$ of $\tau''$ is mapped across any of the new illegal turns 
$(\epsilon _i^{L}, \epsilon _i^{R})$.

\begin{figure}[htb]
\centerline {\includegraphics[scale=0.5]{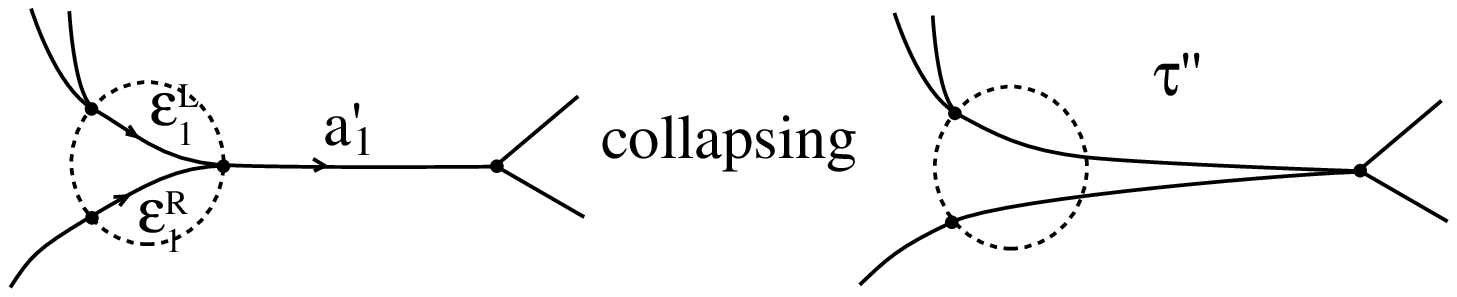}}
\label{fig:7}
\caption{ Splitting at a separatrix.}
\end{figure}

The other orbits of infinitesimal edges of $\tau'$ remains infinitesimal for 
$ (\tau'' , \widetilde{\Psi ''})$. Collapsing the infinitesimal edges of $\tau''$
that bounds $k$-gones or $k$-gones with one side missing to a point leads to an efficient representative $\mathcal{E}'' = (\Gamma'', \Psi'', h_{\Gamma''})$
so that the construction of paragraph 2.2, applied to $\mathcal{E}''$ leads back to $ (\tau'' , \widetilde{\Psi ''})$.

\begin{proposition}
The operation that transforms $\mathcal{E}$ to the efficient representative 
$\mathcal{E}''$ is called a {\em splitting at the separatrix $u_1$}, where $u_1$
corresponds to the infinitesimal turn $(\epsilon _1^{L}, \epsilon _1^{R})$
of  $ (\tau , \widetilde{\Psi }) (\mathcal{E})$, under the map $u_{\mathcal{E}}$ of lemma 2.2.
In addition the cut length $\mathcal{L}_{\mathcal{E}''}$ is non zero for all separatrices for which $\mathcal{L}_{\mathcal{E}}$ is non zero and the separatrices in the orbit $\mathcal{O}$ of $u_1$.
\end{proposition}

The fact that $\mathcal{E}''$ is efficient was proved during the construction. The cut length properties come from the the following observations: \\
Each non trivial periodic Nielsen path of $\mathcal{E}$ has a non trivial image in $\mathcal{E}''$, therefore with non zero length. The lengths are a priori different in $\mathcal{E}$ and in $\mathcal{E}''$ because of the preliminary foldings.\\
The paths $\overline{\epsilon _i^{L}}. \epsilon _i^{R}, i = 1, ..., q$ are periodic Nielsen paths in $\mathcal{E}''$ of non zero length which proves the last statement. $\square$

{\em Remark:} A notion that is similar to this {\em splitting at a separatrix} has been defined in the context of free group automorphisms, for instance in  Los and Lustig \cite{LoLu}, such an operation is called  "unfolding Nielsen paths".

\subsection{The roots and symmetry algorithm.}

We have now defined all the necessary tools to make explicit an algorithm that produces all the cycles of SC-representatives relative to one separatrix in each $\sigma$-orbits and to identify combinatorially and/or metrically the a priori different cycles. These algorithms altogether with the results of section 5 give our solution to the roots and the symmetry problems.

{\verb"Step 0: Find one efficient representative."}\\
From the train track algorithm of Bestvina and Handel \cite{BH95} we find one efficient representative 
$\mathcal{E}_0$ for $[f]$.

{\verb"Step 1: Identify the orbits of separatrices."}\\
The map $u_{\mathcal{E}_0}$ of lemma 2.2 defines an identification of the separatrices $\{u_1, ..., u_m \}$ with some of the illegal turns and some of the infinitesimal turns, ie the set $T^0_{\mathcal{E}_0}$. We label these turns with the corresponding label of the separatrices and we compute the permutation $\sigma$, in particular its order $q$.

{\verb"Step 2: Split all orbits of separatrices."}\\
If some $\sigma$-orbit $\mathcal{O}$ has cut length zero then we apply a splitting operation (paragraph 6.3) at $u_i \in \mathcal{O}$.
Observe that the zero cut length property of a $\sigma$-orbit does not require to compute a length function. After splitting all the orbits with zero cut length
we obtain an efficient representative $\mathcal{E}_{base}$ where all orbits $\mathcal{O}$ of $\sigma$ have positive cut length.
This subclass of efficient representatives has been called "principal blow up class" in Los and Lustig \cite{LoLu}. 

We consider this particular representative $\mathcal{E}_{base}$ as a "base point" for the next steps. In particular we  fix on $\mathcal{E}_{base}$ a length function from which all the other length functions will be derived. The length function is (see paragraph 2.2) an eigenvector of the incidence matrix. 

{\verb"Step 3: Find one SC-representative."}\\
Fix one $\sigma$-orbit  $ \mathcal{O}_i$ and then apply the glueing operation of paragraph 6.1 at all the $\sigma$-orbits  $ \mathcal{O}$ except
$\mathcal{O}_i$. We obtain an efficient representative $\mathcal{E}'_{(i)}$ that satisfies,  by proposition 6.2 :
$\mathcal{L}_{\mathcal{E}'_{(i)}}(t_j) \neq 0 $ iff $j\in \mathcal{O}_i$ but is not necessarily  a SC-representative. 

From the proof of lemma 4.3 there is a vertex with at least 3 gates along one of the Nielsen paths that are associated with the orbit of non zero cut length.
 In addition such a vertex $v_0$ exists on a Nielsen path $N_{t}$, where $t$ is illegal of order one.
We apply a sequence of standard folding operations at $t$ up to the vertex $v_0$. Then we fold at the possible other illegal turns of order one in this orbit until all of them are of order greater than one. Property (SC1) of proposition 3.3 is then satisfied. Property (SC2) is obvious to obtain on boundary loops by blowing up a vertex if it has more than 3 gates.\\
After this step 3 we obtain one SC-representative $\mathcal{E}^1_{(i_1)}$, ..., 
$\mathcal{E}^1_{(i_k)}$, relative to $u_{i_1} \in \mathcal{O}_1$, ...,
$u_{i_k} \in \mathcal{O}_k$.

{\em Remark:}  All SC-representatives $\mathcal{E}^1_{(i_j)}$ are obtained from the base point $\mathcal{E}_{base}$ by a sequence of operations that only consists in gluing and folding. All these elementary operations induce a well defined length function.  Therefore all the $\mathcal{E}^1_{(i_j)}$ have a well defined length function (not up to rescaling).

{\verb"Step 4: Find the cycles."}\\
For each $\mathcal{E}^1_{(i_j)}, j = 1...k$ we apply the elementary SC-folding operations of paragraph 6.2: 
$\mathcal{E}^1_{(i_j)} \rightarrow \mathcal{E}^2_{(i_j)}  \rightarrow ...
 \rightarrow \mathcal{E}^{p_j}_{(i_j)}$, up to the first SC-representative that is combinatorially equivalent to the initial SC-representative: $\mathcal{E}^{p_j}_{(i_j)} \simeq \mathcal{E}^1_{(i_j)}$.
At the end of this step 4 we obtain the collection of all the primitive cycles :
$\mathcal{C}'_{u_1}, ..., \mathcal{C}'_{u_k}$, one for each orbit 
$\mathcal{O}_1, ..., \mathcal{O}_k$ of $\sigma$.

{\em Remark.} The combinatorial equivalence 
$(\Gamma, \Psi) \simeq (\Gamma ', \Psi ')$ is easy to verify. First we check that the two graphs are homeomorphic and then that the two maps are identical after the graph homeomorphism that induces a renaming. We just have to be careful that the graph homeomorphism preserves the surface structure, ie that the cyclic ordering at each vertex is preserved.

{\verb"Step 5: Comparisons of cycles."}\\
We check whether two primitive cycles $\mathcal{C}'_{u_i} \simeq \mathcal{C}'_{u_j}$ where $u_i$ and $u_j$ belong to different $\sigma$-orbits. This is quite easy by the previous remark.

{\verb"Step 6:  Length comparisons."}\\
Fixing a length function at step 2 for the efficient representative $\mathcal{E}_{base}$ is a normalisation. All SC-representatives in the cycles $\mathcal{C}_{u_i}$ are obtained from 
$\mathcal{E}_{base}$ by a sequence of foldings and glueings. By the remark in step 3 all length functions are then well defined.

$\bullet$ For one cycle $\mathcal{C}'_{u_i}$: an elementary SC-folding is a composition of standard foldings and the total length is changed according to proposition 4.2 and lemma 6.3.
At the end of the primitive cycle $\mathcal{C}'_{u_i}$, ie when 
$\mathcal{E}^{p_j}_{(i_j)} \simeq \mathcal{E}^1_{(i_j)}$, and $p_j$ is minimal, we compute the total length $ l ( \mathcal{E}^{p_j}_{(i_j)})$ and we compare it with $\frac{1}{\lambda^q} l ( \mathcal{E}^{1}_{(i_j)})$, as in theorem 4.9.

Observe that $\lambda > 1$ and $q$ are known and, since $\lambda > 1$, only a finite precision is sufficient for the comparison.
This enables to check whether the cycle $\mathcal{C}_{u_i}$ of theorem 5.1 is primitive or not. If not then $\mathcal{C}_{u_i}$ is a finite concatenation of the primitive cycle $\mathcal{C}'_{u_i}$ leading to the existence of a root via one of the propositions 5.3, 5.5. \\

$\bullet$ If step 5 gives two equivalent cycles 
$\mathcal{C}'_{u_i} \simeq \mathcal{C}'_{u_j}$ then we need to compare the lengths between two equivalent SC-representatives $\mathcal{E}^{k}_{(u_j)}$
and $\mathcal{E}^{p}_{(u_i)}$. The equality of the length function implies the existence of a symmetry by proposition 5.7. The non equality shows the existence of a root via one of the propositions 5.4 or 5.6.
As above the precision required for checking the equality/non equality is finite.

Now all the ingredients are computed to check which of the propositions 5.3--5.7 applies. Therefore we have decided if our element $[f]$ admits or not a root or a symmetry.

{\verb"Step 7:  Compute a root or a symmetry."}\\
$\bullet$ The symmetry given by proposition 5.7 is easier to make explicit. It is induced by the combinatorial equivalence $C: \Gamma \rightarrow \Gamma$
of proposition 5.7 between two SC-representatives in different cycles. \\
$\bullet$ From proposition 5.2, a root is given by the transition map between two equivalent SC-representatives. 
In practise this means to follow, along the cycles, how the embedding $h_{\Gamma}$ changes. This is possible but not very simple since it would require to express the embedding in a combinatorial way. 
There is a much simpler way with the tools we already have. 
The most complicated situation is when two primitive cycles are equivalent.
We start at some SC-representative $(\Gamma, \Psi, h_{\Gamma}) = \mathcal{E}^{1}_{(u_i)} \in \mathcal{C}_{u_i}$ and then apply the uniquely defined sequence of foldings given by the primitive cycle followed by the (unique) equivalence homeomorphism obtained by 
$\mathcal{E}^{1}_{u_i} \simeq \mathcal{E}^{k}_{u_j}$. 
This sequence of foldings followed by a graph homeomorphism defines a unique efficient representative $\phi: \Gamma \rightarrow \Gamma$ of $[g]$ that is a root of $[f]$. 
This completes the proof of the main theorem.$\square$

\section{Example}

In this section we present only one simple example in the mapping class group of the punctured disc that is classically given as the braid group modulo the centre.
More complicated examples would require much more space and even for this simple example we give only a few steps. 

Consider the braid $\beta\in B_4$,  given with the classical braid generators by $\beta=\sigma_3\overline{\sigma_2}\sigma_1$.  It defines a pseudo-Anosov element in the mapping class group of the 4th punctured disc. An efficient representative is given on the graph presented in figure 8, by the following combinatorial map:
$$\begin{array}{ccc}
 a & \longmapsto & a t_1 \overline{a}b\\
 b  & \longmapsto & ec\\
 c  & \longmapsto & d\\
 d  & \longmapsto & c\overline{t_3}\overline{c}\overline{e}a\\
 e  & \longmapsto & ec\overline{t_3}\overline{c}\overline{e}a t_1\overline{a}ect_3\overline{c}\ .
\end{array}$$


\begin{figure}[htb]
\centerline {\includegraphics[height=40mm]{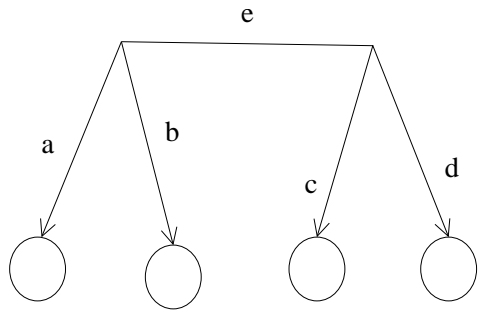}}
\label{fig:graphini}
\caption{  The graph of the initial efficient representative.}
\end{figure}

It is not a single cut representative since there are 2 fixed illegal turns: $(b,e)$ and $(d,\overline{e})$, corresponding to 2 unstable separatrices that are fixed under the homeomorphism, respectively $u_1$ and $u_2$. This efficient representative will be our "base" representative $\mathcal{E}_{base}$ of step 2 above.

{\bf Glueing the separatrix $u_1$ } ( turn $(d,\overline{e})$ ): \\
This glueing operation will be presented as a sequence of classical foldings.\\
step 1 - folding $(\overline{e},d)$, then $(\overline{e},\overline{t_4})$, then $(\overline{e},\overline{d})$.\\ 
After this sequence of foldings we obtain the following graph,  embedding and map:

\begin{figure}[htb]
\centerline {\includegraphics[height=40mm]{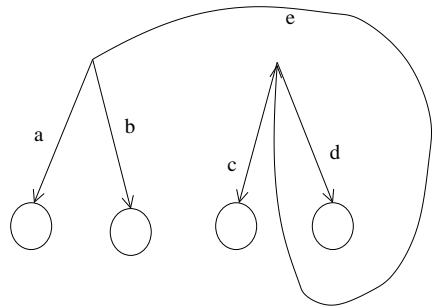}}
\label{fig:graph1}
\caption{  First folding .}
\end{figure}

$$\begin{array}{ccc}
 a & \longmapsto & a t_1 \overline{a}b\\
 b  & \longmapsto & edt_4\overline{d}c\\
 c  & \longmapsto & d\\
 d  & \longmapsto & c\overline{t_3}\overline{c}d\overline{t_4}\overline{d}\overline{e}a\\
 e  & \longmapsto & edt_4\overline{d}\ .
\end{array}$$

Then we apply four similar steps that are uniquely defined since there is only one turn to be folded at each step. 
After these four steps we obtain the graph, embedding and map as shown in the next figure.

\begin{figure}[htb]
\centerline {\includegraphics[height=40mm]{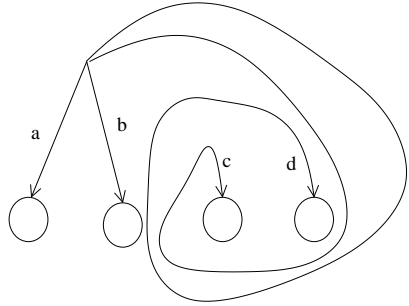}}
\label{fig:graph4}
\caption{  A single cut representative.}
\end{figure}

 $$\begin{array}{ccc}
 a & \longmapsto & a t_1 \overline{a}b\\
 b  & \longmapsto & ct_3\overline{c}dt_4\overline{d}c\\
  c  & \longmapsto & c\overline{t_3}\overline{c}d\\
d & \longmapsto & d\overline{t_4}\overline{d}c\overline{t_3}\overline{c}a
\end{array}$$
This representative is single cut, relative to $u_2$.

{\bf Glueing the separatrix $u_2$ } ( at the turn $(b,e)$).\\
Analogous operations show that a single cut representative relative to the other separatrix is given on the following graph and embedding:
\begin{figure}[htb]
\centerline {\includegraphics[height=40mm]{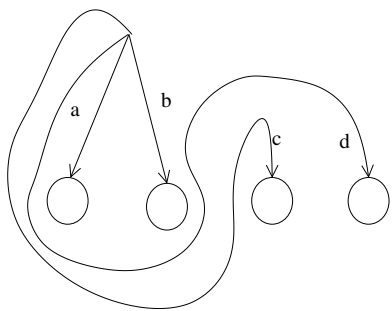}}
\label{fig:graphtgc2}
\caption{Another single cut representative.}
\end{figure}

 $$\begin{array}{ccc}
 a & \longmapsto & a t_1 \overline{a}b\\
 b  & \longmapsto & b\overline{t_2}\overline{b}at_1\overline{a}c\\
  c  & \longmapsto & ct_3\overline{c}d\\
d & \longmapsto & at_1\overline{a}bt_2\overline{b}a 
\end{array}$$
The two graphs of figure 10 and 11 are obviously homeomorphic and the two representatives are combinatorially equivalent under the map:  $(a,b,c,d)\mapsto(c,d,b,a)$. 
We check that this equivalence is realised by a homeomorphism of the punctured disc that is induced by the braid: $\alpha=\sigma_3^2\sigma_2\sigma_1\sigma_3\sigma_2$. 
It follows from the main theorem that 
$$\alpha \beta=\beta\alpha,$$
this equality is easy to check directly.
Since $\alpha^2=\Delta$, where $\Delta$ is the classical Garside braid generating the centre, then $\alpha$ is finite order.
Observe that for this example we didn't compute the whole cycles of theorem 5.1 since the first SC-representatives relative to each fixed separatrix were already equivalent.

%
%
%
%

\end{document}